\documentclass[]{article}

\newcommand{\un}{\text{ $\underline{\triangleright}$ }}
\newcommand{\ovr}{\text{ $\overline{\triangleright}$ }}
\usepackage[none]{hyphenat}
\usepackage{xcolor, array, amssymb}
\usepackage{amsmath,amsfonts,graphicx,amsthm,geometry,setspace,subcaption,overpic}
\usepackage{etoolbox,tikz-cd,rotating,url}

\theoremstyle{definition}
\newtheorem{theorem}{Theorem}[subsection]
\newtheorem{example}{Example}[subsection]
\newtheorem{definition}{Definition}[subsection]
\newtheorem{proposition}{Proposition}[subsection]

\newtheorem{corollary}{Corollary}[subsection]
\newtheorem{remark}{Remark}[subsection]

\title{Persistent Homology of Biquandle Coloring Quivers}
\author{Hamdi Kayaslan}

\begin{document}
	\maketitle
    \onehalfspacing
	\begin{abstract}
        In this paper, we extend the notion of directed clique complex to quivers and introduce an associated homology theory. By applying this construction to biquandle coloring quivers, we obtain new invariants of links. We then introduce a quiver filtration-valued invariant of links induced by filtrations of biquandle endomorphism sets. We construct persistent homological invariants of links by applying persistence techniques to these quiver filtrations through the introduced homology theory. 
	\end{abstract}

\section{Introduction}

Knot theory studies embeddings of circles in $S^3$ and more generally collections of such circles, called \textit{links}. In addition to the classical setting, virtual knot theory, introduced by Louis H. Kauffman in \cite{kauffman2012introduction}, studies links in thickened surfaces. These theories admit diagrammatic descriptions on surfaces in which crossings encode the interaction of strands. This description provides a combinatorial setting for defining and studying invariants of classical and virtual links.

Many invariants of classical and virtual links arise from assigning algebraic data to the link diagrams which are preserved under the underlying equivalence moves, called the \textit{Reidemeister moves}. Among these, \textit{biquandles} -whose axioms encode Reidemeister moves- were introduced in \cite{fenn1995trunks} and provide computable tools for constructing invariants for classical and virtual links, see \cite{fenn2004biquandles, hrencecin2007biquandles,nelson2017quantum,nelson2019biquandle,nelson2024quandle}. In particular, colorings of link diagrams by finite biquandles give rise to counting invariants of classical and virtual links.

An invariant $I'$ is called an \textit{enhancement} of an invariant $I$ if it retains the information of $I$. The enhancement is called \textit{proper} if $I'$ distinguishes links that are not distinguished by $I$. Biquandle counting invariants have been enhanced through additional algebraic and combinatorial constructions. One such proper enhancement is the \textit{biquandle coloring quiver}, introduced in \cite{ceniceros2023psyquandle} analogously to quandle coloring quivers \cite{cho2019quandle}, to enrich the counting invariant by utilizing biquandle endomorphisms. In the same paper, the \textit{in-degree quiver polynomial invariant} of links was introduced using biquandle coloring quivers.

While biquandle coloring quivers provide a richer invariant than the biquandle counting invariant, their structure as directed graphs suggests the possibility of obtaining new invariants of links as it was pointed out in \cite{cho2019quandle}. In particular, it is natural to ask what kind of new and computable invariants of links can be obtained by applying homological tools of algebraic topology to biquandle coloring quivers.

The topology of directed graphs is of great interest and has been studied extensively; see, for example, \cite{caputi2024hochschild, grigoryan2015cohomology, hepworth2025reach, ivanov2024simplicial}. In particular, combinatorial constructions such as the \textit{directed clique complex} (also referred to as the \textit{directed flag complex}) provide a way to associate higher-dimensional topological structures to directed graphs. This perspective was introduced in the context of network topology in \cite{masulli2016topology} and has since been utilized in applications, for instance in the study of neuronal networks \cite{reimann2017cliques}, where higher-order connectivity patterns are encoded via simplicial complexes.

However, such constructions are typically defined for simple directed graphs and do not directly incorporate multiplicities of edges, and hence are not immediately applicable to quivers. On the other hand, quivers and directed multigraphs have been studied from a homological perspective in \cite{grigoryan2018path} via \textit{path homology}, where multiplicities of directed edges play a central role in the theory.

Motivated by these two perspectives, we introduce a generalization of the directed clique (flag) complex to quivers by incorporating edge multiplicities through a threshold condition. More precisely, we define the \textit{$N$-directed clique complex}, where $N \geq 1$ is an integer, by requiring the existence of at least $N$ directed edges between vertices in the given orientation. This leads to a corresponding notion of \textit{$N$-directed clique homology} of quivers. By applying this construction to biquandle coloring quivers, we obtain a new homological invariant of links, where the strength and computational complexity of the invariant can be adjusted through the choice of $N$. 

\begin{theorem}
	\label{main1}
	Let $L$ be a link, $X$ a finite biquandle, $S\subseteq$ End$(X)$, and $N\geq 1$. Then, for all $n\geq 0$ and $N\geq1$, the $n$-th $N$-directed clique homology group $H_n^{(N)}(L;X,S)$ is an invariant of links.
\end{theorem}

A \textit{filtration} of an object is a sequence of its subobjects forming an increasing chain. We enhance the biquandle coloring quiver invariant by introducing the \textit{biquandle coloring quiver filtration} of links obtained from filtrations of biquandle endomorphism sets.

\begin{theorem}
	\label{maint2}
	Let $L$ be a link, $X$ a finite biquandle, and $S_\ast$ a filtration of End$(X)$. The $X$-coloring quiver filtration $\mathcal{Q}_X^{S_\ast}(L)$ is an invariant of links.
\end{theorem}
\noindent This invariance ensures that any functorial construction applied to the biquandle coloring quiver filtration yields an invariant of links. In particular, applying quiver invariants to each stage produces ordered multiset invariants of links. For instance, we apply the in-degree polynomial and obtain the \textit{ordered in-degree polynomial set} invariant of links.

Persistent homology studies the evolution of homological features across filtrations of topological spaces and it is a central tool of topological data analysis. It was introduced in \cite{edelsbrunner2002topological}, and was further developed algebraically and computationally in \cite{zomorodian2005computing}. Since then, it has been studied extensively in \cite{carlsson2009topology, carlsson2010zigzag, chazal2016structure, cohen2007stability} and many more. The \textit{birth} and \textit{death} of homological features along a filtration determine intervals, called \textit{persistence intervals}. The multiset of these persistence intervals is called \textit{persistence barcode}, see \cite{ghrist2008barcodes}. In many applications, short persistence intervals are considered as noise and discarded, while longer intervals signal significant topological features.

Persistence provides a powerful homological summary of filtered topological data. However, it is not, in general, an invariant of topological spaces. The reason is that the filtering methods are sensitive to deformations of topological spaces and may result in non-isomorphic filtrations in general. Hence, instead, the strength of persistence lies in its stability. That is, under suitable perturbations of the input filtration, the difference on the resulting persistence barcodes are small. This stability property makes persistent homology effective in applications where data includes noise. See \cite{chazal2016structure, cohen2007stability}.

Persistent homology has recently been extended to directed networks through a variety of approaches; see, for example, \cite{aktas2019persistence, carstens2013persistent, dey2022efficient, lin2019weighted}. In particular, in \cite{lutgehetmann2020computing}, persistent homology of directed flag complexes was studied, providing a framework for computing persistence barcodes associated to directed graphs via simplicial constructions. Another approach in \cite{chowdhury2018persistent} introduces \textit{persistent path homology}, which associates stable persistence invariants to weighted directed networks. More recently, alternative approaches have been developed to incorporate directionality more intrinsically, such as the directed persistent homology theory for dissimilarity functions introduced in \cite{mendez2023directed}, where asymmetry is encoded at the level of the underlying data.

These constructions demonstrate that persistent homology can be successfully adapted to directed settings, either by modifying the underlying simplicial complexes or by altering the algebraic framework of homology. However, most existing approaches are formulated for directed graphs without multiplicities or arise from metric or weight-based data.

In contrast, the approach taken in this paper is combinatorial in nature and is based on quivers, where multiplicities of directed edges play a fundamental role.
We incorporate the $N$-directed clique homology construction into a persistence framework by considering filtrations of biquandle coloring quivers induced by filtrations of biquandle endomorphism sets. This enables us to associate persistent homology groups and barcodes to links in a natural and diagrammatically invariant way (up to generalized Reidemeister moves).

\begin{theorem}
	\label{main2}
	Let $L$ be a link, $X$ a finite biquandle and $S_\ast$ a filtration of End$(X)$. For $i \leq j$, the $n$-th persistent $N$-directed clique homology group $H_{n}^{(N),i,j}(L;X,S_\ast)$ is an invariant of links for $n\geq 0$ and $N\geq1$.
\end{theorem}

Moreover, in contrast to classical applications of persistent homology, we show that all persistence intervals carry meaningful information in our combinatorial setting. In particular, even the shortest intervals may distinguish non-equivalent links. Additionally, we introduce an invariant that records the multiplicities of instantly vanishing homology classes in the filtration as matrices which we call \textit{stillborn matrices}. We define the \textit{persistence pair invariant} of links as the pair of persistence barcode and stillborn matrix. We show that the persistence pair invariant is a proper enhancement of both persistence barcode invariant and stillborn matrix invariant.

This work applies homological techniques from topological data analysis to link invariants that are based on biquandles and biquandle colorings of links. This establishes a bridge between knot theory and topological data analysis which may lead to new directions for further research.

The organization of the paper is as follows. In Section~\ref{pre}, we briefly review the ingredients we use in the body of the paper. In particular, we give basics of classical and virtual links in Section~\ref{cvlinks}. In Section~\ref{biqus}, we review biquandles and invariants arising from biquandle colorings. In Section~\ref{homologies}, we mention directed clique complexes of directed graphs. Section~\ref{persistenthom} covers the basics of persistent homology. We present our results in Section~\ref{mainsection1}. In Section~\ref{ndirectsec}, we introduce $N$-directed clique complexes, study their homology and prove Theorem~\ref{main1}. In Section~\ref{filters}, we introduce biquandle coloring quiver filtrations of links, show its invariance by proving Theorem~\ref{maint2} and point out its relations with the biquandle coloring quiver. In Section~\ref{persistsec}, we apply the persistence approach through the $N$-directed clique homology to biquandle coloring quiver filtrations and prove Theorem~\ref{main2}. We obtain the persistence barcode invariant of links as an enhancement of the biquandle counting invariant. We then enhance the persistence barcode invariant using the stillborn matrices and prove that the enhancements in this section are proper by giving concrete examples. We conclude with an appendix briefly explaining the Python code that we use in our computations, which is available on Github.

\section{Preliminaries}
\label{pre}
\subsection{Classical and virtual links}
\label{cvlinks}
Knot theory studies embeddings of circles in three-dimensional spaces and includes both classical knot theory and its extension, virtual knot theory. More generally, one considers \textit{links}, which consist of finitely many disjoint embedded circles in three-dimensional spaces. The theory of such embeddings in $S^3$ is called the \textit{classical link theory}, where the theory of embeddings in thickened surfaces is called the \textit{virtual link theory}. A \textit{knot} is a link consisting of a single component. 

These theories can be studied via diagrammatic representations on surfaces. A \textit{classical link diagram} is an immersion of finitely many copies of the unit circle $S^1$ in the plane with finitely many self-intersections where each self-intersection is transversal and endowed with under- or over-passing information. Such self-intersections are called \textit{classical crossings} and illustrated in Figure~\ref{cc}. See Figure~\ref{ck} for an example of a classical knot diagram (a single component classical link diagram) and Figure~\ref{cl} for a classical link diagram with two components. A \textit{classical link} is defined as an equivalence class of classical link diagrams up to the \textit{classical Reidemeister moves} shown in Figure~\ref{vcrm}; we refer to this equivalence as \textit{classical equivalence}.

\begin{figure}
	\centering    
	\begin{subfigure}{0.5\textwidth}
		\centering
		\includegraphics[width=33mm]{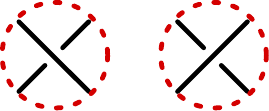}
		\caption{Classical crossings.}
		\label{cc}
	\end{subfigure}\hfill
	\begin{subfigure}{0.5\textwidth}
		\centering
		\includegraphics[width=14mm]{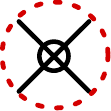}
		\caption{Virtual crossing.}
		\label{vc}
	\end{subfigure}
	\caption{}
	\label{virtualknotoidmoves}
\end{figure}

\begin{figure}
	\centering 
	\begin{subfigure}{0.5\textwidth}
		\centering
		\includegraphics[width=27mm]{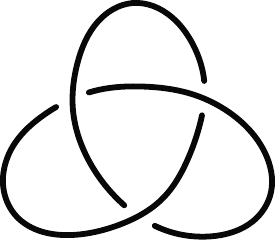}
		\vspace{0.3 cm}
		\caption{A classical knot diagram.}
		\label{ck}
	\end{subfigure}\hfill
	\begin{subfigure}{0.5\textwidth}
		\centering
		\includegraphics[width=23mm]{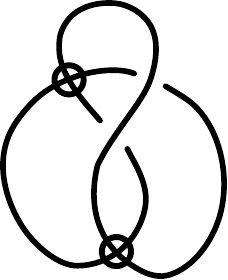}
		\vspace{0.3 cm}
		\caption{A virtual knot diagram.}
		\label{vk}
	\end{subfigure}
	\begin{subfigure}{0.5\textwidth}
		\centering
		\includegraphics[width=30mm]{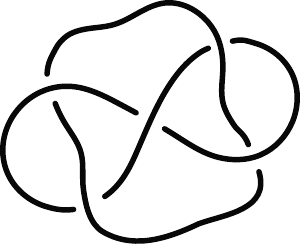}
		\caption{A classical link diagram.}
		\label{cl}
		\vspace{0.3 cm}
	\end{subfigure}\hfill
	\begin{subfigure}{0.5\textwidth}
		\centering
		\includegraphics[width=28mm]{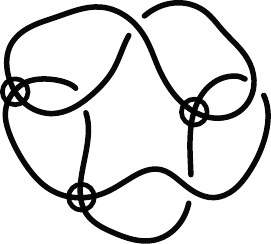}
		\caption{A virtual link diagram.}
		\vspace{0.3 cm}
		\label{vl}
	\end{subfigure}
	
	\caption{}
	\label{virtualknotoidmoves}
\end{figure}

\begin{figure}
	\centering    
	\begin{subfigure}{0.5\textwidth}
		\centering
		\includegraphics[width=40mm]{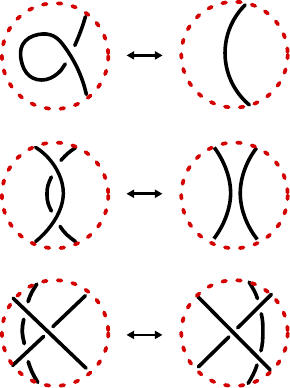}
		\caption{Classical Reidemeister moves.}
		\label{vcrm}
		\vspace{0.3 cm}
	\end{subfigure}\hfill
	\begin{subfigure}{0.5\textwidth}
		\centering
		\includegraphics[width=40mm]{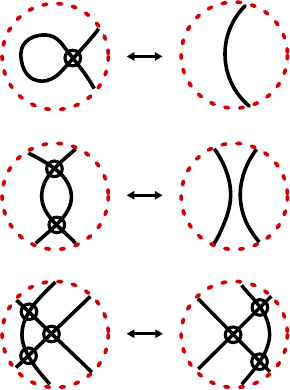}
		\caption{Virtual Reidemeister moves.}
		\vspace{0.3 cm}
		\label{vvrm}
	\end{subfigure}
	\begin{subfigure}{0.5\textwidth}
		\centering
		\includegraphics[width=40mm]{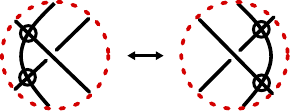}
		\vspace{0.1 cm}
		\caption{Mixed Reidemeister move.}
		\label{vmm}
	\end{subfigure}\hfill
	\caption{Generalized Reidemeister moves.}
	\label{generalizedRmoves}
\end{figure}

Virtual link diagrams extend classical link diagrams by allowing an additional type of crossing called a \textit{virtual crossing}. A virtual crossing is a transversal self-intersection without over- or under-passing information and is indicated by a small circle enclosing the crossing (see Figure~\ref{vc}). A \textit{virtual link diagram} is a link diagram containing finitely many classical and/or virtual crossings. See Figure~\ref{vk} for an example of a virtual knot diagram (a single component virtual link diagram) and Figure~\ref{vl} for an example of a virtual link diagram with two components. A \textit{virtual link} is defined as an equivalence class of virtual link diagrams up to the \textit{generalized Reidemeister moves} shown in Figure~\ref{generalizedRmoves}; we refer to this equivalence as \textit{virtual equivalence}.

In this paper, we study invariants that are based on labelings of semi-arcs of oriented (classical or virtual) link diagrams. A \textit{semi-arc} of a link diagram is a segment of the diagram between two consecutive classical crossings. For instance, the diagrams in Figures~\ref{ck}, \ref{vk}, \ref{cl} and \ref{vl} have six, four, ten and eight semi-arcs, respectively. An \textit{oriented link diagram} is a link diagram together with an \textit{orientation}, that is, a choice of direction along each component of the link diagram indicated for each component by an arrow on a semi-arc of the component.

Every classical link diagram can be regarded as a virtual link diagram without virtual crossings. It is therefore natural to ask whether classical equivalence classes of classical links coincide with their virtual equivalence classes. In other words, does classical link theory embed into virtual link theory? This question was answered affirmatively by Greg Kuperberg in \cite{kuperberg2003what}. Consequently, any invariant of virtual links restricts to an invariant of classical links. For this reason, we work with virtual links. Throughout the paper, the term link will refer to an oriented virtual link unless otherwise specified. In particular, a (classical or virtual) knot will mean a single component (virtual) link considered up to virtual equivalence.

\subsection{Biquandle coloring quivers}
\label{biqus}
We begin with the definition of a biquandle and proceed by providing some examples of biquandles.
\begin{definition} \cite{nelson2017quantum}
	A set $X$ with two operations $\un,\ovr:X\times X\rightarrow X$ such that for all $x,y,z\in X$,
	\begin{itemize}
		\item[(1)] $x\un x=x\ovr x$,
		\item[(2)] the maps
		\begin{align*}
			&\alpha_y:X\rightarrow X, \qquad \beta_y:X\rightarrow X,\qquad S:X\times X\rightarrow X\times X,\\
			&\quad x\mapsto x\ovr y\qquad x\mapsto x\un y\qquad\quad (x,y)\mapsto (y\ovr x, x\un y)
		\end{align*}
		are invertible,
		\item[(3)] and the exchange laws hold, \begin{align*}
			(x\un y)\un(z\un y)=(x\un z)\un(y\ovr z),\\
			(x\un y)\ovr(z\un y)=(x\ovr z)\un(y\ovr z),\\
			(x\ovr y)\ovr(z\ovr y)=(x\ovr z)\ovr(y\un z),
		\end{align*}
	\end{itemize}
	is called a \textit{biquandle}, denoted by $(X,\un,\ovr)$. In this paper, we use $X$ to denote a biquandle for notational convenience.
	\label{bq}
\end{definition}


\begin{example}
	\label{biqex}
	Let $A$ be a module over the ring $\mathbb{Z}[t^{\pm1},r^{\pm1}]$ and define operations $\un,\ovr:A\times A\rightarrow A$ by
	\[ x\un y=tx+(r-t)y\quad\text{ and }\quad x\ovr y=r
	x\]
	for all $x,y\in A$.
	$(A,\un,\ovr)$ is a biquandle, called the \textit{Alexander biquandle.} In particular, if $A$ is a commutative ring, one can construct an Alexander biquandle by choosing two invertible elements $t,r\in A$.
\end{example}

It is usually more convenient to express the operations of a finite biquandle using matrices. Let $X=\{x_1,x_2,...,x_n\}$ be a biquandle. Then the \textit{biquandle operation matrix} of $X$ is the augmented $n\times 2n$ matrix $M(X)=(a_{i,j})$ defined by, \[a_{i,j}= k \quad \text{if} \quad x_k=\begin{cases}
	x_i\un x_j, & 1\leq j\leq n\\
	x_i\ovr x_{j-n}, & n+1\leq j\leq 2n.
\end{cases}\]

\begin{example}
	\label{alex}
	Let $A=\mathbb{Z}_2[s]/\langle s^2+s+1\rangle$ and $t=s$, $r=s+1$. Then $A$, equipped with the operations
	\[ x\un y= sx+y,\quad x\ovr y= (s+1)x,\]
	is an Alexander biquandle. If we label the elements of $A$ by, \[x_1=0,\;x_2=1,\;x_3=s,\;x_4=s+1,\] then the biquandle operation matrix of $A$ is,
	\[M(A)=\left[\begin{array}{c c c c | c c c c}
		1& 2& 3& 4& 1& 1& 1& 1\\
		3& 4& 1& 2& 4& 4& 4& 4\\
		4& 3& 2& 1& 2& 2& 2& 2\\
		2& 1& 4& 3& 3& 3& 3& 3
	\end{array}\right].\]
\end{example}

\begin{definition} \cite{nelson2017quantum}
	Let $L$ be a link diagram whose semi-arcs are labeled and let $G(L)$ denote the set of labels of the semi-arcs of $L$. Let $R(L)$ denote the set of relations among the labels of the semi-arcs at the crossings of $L$, as shown in Figure \ref{rbqq}. The \textit{fundamental biquandle} of $L$, denoted by $\mathcal{B}(L)$, is the biquandle generated by the elements of $G(L)$ subject to the relations $R(L)$. That is,
	\[\mathcal{B}(L)=\langle G(L)\;|\; R(L)\rangle.\]
\end{definition}

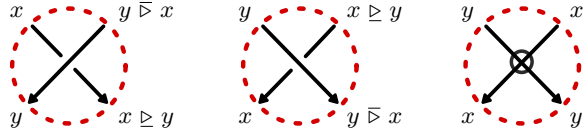
\begin{figure}
	\centering
	\begin{overpic}[width=0.5\linewidth]{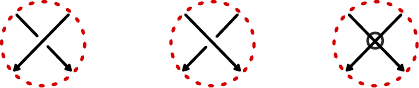}
		\put(0,19){\small$x$}
		\put(0,0.5){\small$y$}
		\put(19,0.5){\small$x\un y$}
		\put(19,19){\small$y\ovr x$}
		
		\put(40,19){\small$y$}
		\put(40,0.5){\small$x$}
		\put(59,0.5){\small$y\ovr x$}
		\put(59,19){\small$x\un y$}
		
		\put(79,19){\small$y$}
		\put(79,0.5){\small$x$}
		\put(98,0.5){\small$y$}
		\put(98,19){\small$x$}
	\end{overpic}
	\vspace{0.5 cm}
	\caption{Local labeling rules for classical and virtual crossings.}
	\label{rbqq}
\end{figure}

\begin{example}
	A virtual knot diagram $L$ with labeled semi-arcs is given in Figure~\ref{lvkd}. In Figure~\ref{crossingrels}, we illustrate the relations determined by the crossings of $L$. Thus, the fundamental biquandle of $L$ admits the presentation,
	\[\mathcal{B}(L)=\langle\, x,y,w,z\;|\; x=z\ovr w,\, y=w\un z,\, y=x\un z,\, w=z\ovr x\,\rangle.\]
	\begin{figure}
		\centering    
		\begin{subfigure}{0.5\textwidth}
			\centering
			\begin{overpic}[width=0.35\linewidth]{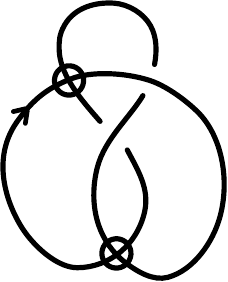}
				\put(34,88){$y$}
				\put(68,37){$z$}
				\put(50,54){$w$}
				\put(7,37){$x$}
				
				\put(7,75){\textcolor{blue}{$c_2$}}
				\put(60,73){\textcolor{blue}{$c_1$}}
				\put(37,37){\textcolor{blue}{$c_3$}}
				\put(37,-5){\textcolor{blue}{$c_4$}}
			\end{overpic}
			\vspace{0.3 cm}
			\caption{}
			\label{lvkd}
			\vspace{0.3 cm}
		\end{subfigure}\hfill
		\begin{subfigure}{0.5\textwidth}
			\centering
			\begin{overpic}[width=0.65\linewidth]{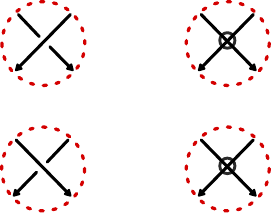}
				\put(-1,74){\small$w$}
				\put(-1,47){\small$z$}
				\put(29,47){\small$y=w\un z$}
				\put(29,74){\small$x=z\ovr w$}
				
				\put(67,74){\small$x$}
				\put(67,47){\small$y$}
				\put(97,47){\small$x$}
				\put(97,74){\small$y$}
				
				\put(-1,28){\small$z$}
				\put(-1,1){\small$x$}
				\put(29,1){\small$w=z\ovr x$}
				\put(29,28){\small$y=x\un z$}
				
				\put(67,28){\small$x$}
				\put(67,1){\small$z$}
				\put(97,1){\small$x$}
				\put(97,28){\small$z$}
				
				\put(81,55){\small\textcolor{blue}{$c_2$}}
				\put(13,55){\small\textcolor{blue}{$c_1$}}
				\put(13,9){\small\textcolor{blue}{$c_3$}}
				\put(81,9){\small\textcolor{blue}{$c_4$}}
			\end{overpic}
			\caption{}
			\vspace{0.3 cm}
			\label{crossingrels}
		\end{subfigure}
		\caption{A virtual knot diagram (a) and its crossing relations (b).}
		\label{funbiq}
	\end{figure}
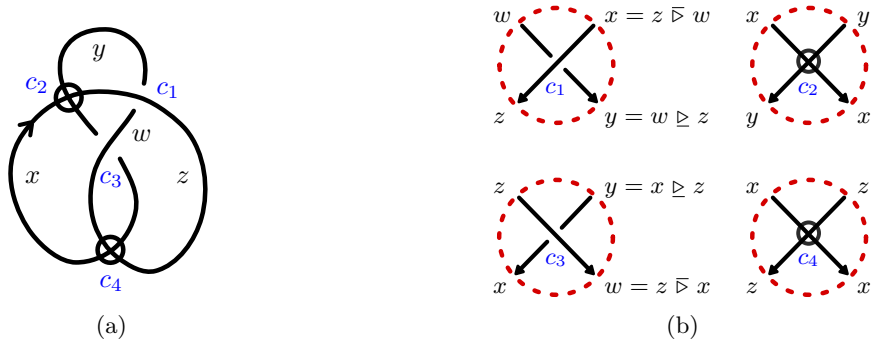
\end{example}

The fundamental biquandle is an invariant of classical and virtual links. Moreover, it is complete up to reflection for classical links. That is, if the fundamental biquandles of two classical links are isomorphic, then either these links are equivalent, or one is the reflection of the other. However, determining whether two fundamental biquandles are isomorphic is a difficult task in general. We refer the reader to \cite{elhamdadi2015quandles, fenn2004biquandles, hrencecin2007biquandles} for more details on the fundamental biquandle.

\begin{definition} \cite{nelson2017quantum}
	Let $(X,\un_{\text{\hspace{-0.08 cm}\tiny{$X$}}},\ovr_{\text{\hspace{-0.08 cm}\vspace{0.1cm}\tiny{$X$}}})$ and $(Y,\un_{\text{\hspace{-0.08 cm}\tiny{$Y$}}},\ovr_{\text{\hspace{-0.08 cm}\tiny{$Y$}}})$ be two biquandles. A biquandle homomorphism is a map $f:X\rightarrow Y$ satisfying 
	\[f(x\un_{\text{\hspace{-0.08 cm}\tiny{$X$}}} \,y)=f(x)\un_{\text{\hspace{-0.08 cm}\tiny{$Y$}}} f(y), \qquad f(x\ovr_{\text{\hspace{-0.08 cm}\tiny{$X$}}} \,y)=f(x)\ovr_{\text{\hspace{-0.08 cm}\tiny{$Y$}}} f(y)\]
	for all $x,y\in X.$
\end{definition}

\begin{definition} \cite{nelson2017quantum}
	Let $X$ be a biquandle and let $L$ be a link diagram. An \textit{$X$-coloring} of $L$ is a labeling of semi-arcs of $L$ by the elements of $X$ such that the labels around each classical or virtual crossing respects the local labeling rules given in Figure~\ref{rbqq}. The set of $X$-colorings of $L$ is denoted by $\text{Col}_X(L)$. 
\end{definition}

\begin{remark}
	\label{colorinv}
	Let $L'$ be the link diagram obtained from a link diagram $L$ by a generalized Reidemeister move. The axioms of the biquandle are exactly the conditions ensuring that an $X$-coloring of $L$ corresponds to a unique $X$-coloring of $L'$. We refer the reader to \cite{nelson2017quantum} for more details and explicit derivations of the axioms of the biquandle.
\end{remark}

\begin{remark}
	Let $X$ be a biquandle and let $L$ be a link diagram. A biquandle homomorphism $f:\mathcal{B}(L)\rightarrow X$ determines an $X$-coloring of $L$. Indeed, label each semi-arc of $L$ by the corresponding generator of $\mathcal{B}(L)$. Then such a homomorphism $f$ gives a valid $X$-coloring of $L$ since the images of the labels under $f$ respect the local labeling rules for classical and virtual crossings as depicted in Figure~\ref{biqhom}. Conversely, each $X$-coloring of $L$ determines a biquandle homomorphism from $\mathcal{B}(L)$ to $X$ in the canonical way. This correspondence allows us to denote an $X$-coloring of $L$ by $f$, and use $L_f$ to denote a diagram $L$ colored by $f$.
\end{remark}

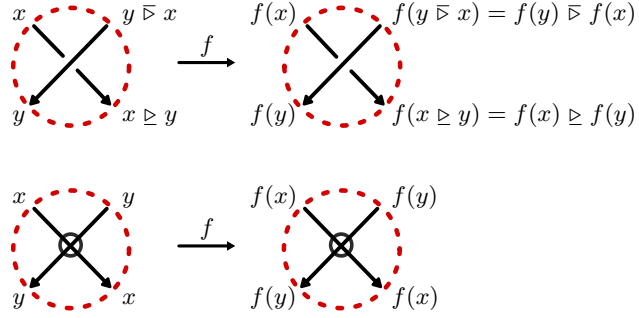
\begin{figure}
	\centering
	\begin{overpic}[width=0.4\linewidth]{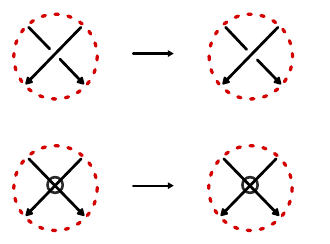}
		\put(4,67){\small$x$}
		\put(4,45){\small$y$}
		\put(28,67){\small$y\ovr x$}
		\put(28,45){\small$x\un y$}
		
		\put(4,27){\small$x$}
		\put(4,5){\small$y$}
		\put(28,27){\small$y$}
		\put(28,5){\small$x$}
		
		\put(45,60){\small$f$}
		\put(45,20){\small$f$}
		
		\put(56,67){\small$f(x)$}
		\put(56,45){\small$f(y)$}
		\put(87,67){\small$f(y\ovr x)=f(y)\ovr f(x)$}
		\put(87,45){\small$f(x\un y)=f(x)\un f(y)$}
		
		\put(56,27){\small$f(x)$}
		\put(56,5){\small$f(y)$}
		\put(87,27){\small$f(y)$}
		\put(87,5){\small$f(x)$}
	\end{overpic}
	\vspace{0.2 cm}
	\caption{Fundamental biquandle colorings at classical and virtual crossings and their images under a homomorphism $f$.}
	\label{biqhom}
\end{figure}

Biquandle colorings have been used to construct several invariants of links. An immediate example is the biquandle counting invariant, which is motivated by Remark~\ref{colorinv}. For a finite biquandle $X$, the \textit{biquandle counting invariant} is defined as the number of $X$-colorings of a link $L$, and is denoted by $\Phi_X^{\mathbb{Z}}(L)$. Note that $\Phi_X^{\mathbb{Z}}(L)=|\text{Col}_X(L)|=|\text{Hom}(\mathcal{B}(L),X)|$, where $\text{Hom}(\mathcal{B}(L),X)$ denotes the set of biquandle homomorphisms from $\mathcal{B}(L)$ to $X$.

Let End$(X)$ denote the set of biquandle endomorphisms of a biquandle $X$. In \cite{ceniceros2023psyquandle}, the biquandle coloring invariant was enhanced by using the observation that an endomorphism $\sigma\in \text{End}(X)$ composed with an $X$-coloring $f\in \text{Hom}(B(L),X)$ is again an $X$-coloring $\sigma\circ f\in \text{Hom}(B(L),X)$ of $L$.

\begin{remark}
	Throught the paper, we work with finite biquandles. If $X=\{x_1,...,x_n\}$ and $f:X\rightarrow Y$ is a biquandle homomorphism, we denote $f$ by the tuple \[ f=[f(x_1), ...,f(x_n)].\]
\end{remark}

\begin{definition} \cite{ceniceros2023psyquandle}
	Let $X$ be a finite biquandle and let $L$ be a link diagram. Associate to each $X$-coloring $f$ of $L$ a vertex $v_f$, and let $V$ denote the set of these vertices. Fix a subset $S\subseteq \text{End}(X)$. For two $X$-colorings $f$ and $g$, let there be a directed edge from $v_f$ to $v_g$, if $\alpha \circ f=g$ for some $\alpha \in S$, and denote the set of these edges by $E$. The \textit{$X$-coloring quiver of $L$}, denoted by $\mathcal{Q}_X^{S}(L)$, is a directed graph with the vertex set $V$ and the edge set $E$ possibly having loops and multi-edges. If $S=\text{End}(X)$, then the $X$-coloring quiver is called the \textit{full $X$-coloring quiver} of $L$, and is denoted by $\mathcal{Q}_X(L)$.
\end{definition}
\begin{theorem} \cite{ceniceros2023psyquandle}
	\label{quivinv}
	Let $X$ be a finite biquandle and $S\subseteq \text{End}(X)$. The $X$-coloring quiver $\mathcal{Q}_X^S(L)$ is an invariant of links.
\end{theorem}

It is clear that for a link diagram $L$ and a finite biquandle $X$, the number of vertices in the biquandle coloring quiver $\mathcal{Q}_X^{S}(L)$ is equal to the number of biquandle colorings $\Phi_X^{\mathbb{Z}}(L)$. Hence, $\mathcal{Q}_X^{S}(L)$ is an enhancement of $\Phi_X^{\mathbb{Z}}(L)$. It is shown in \cite{ceniceros2023psyquandle} that this enhancement is proper. 

\begin{example}
	\label{indeg}
	Let $L$ be the virtual knot diagram shown in Figure~\ref{lvkd} and $X=\mathbb{Z}_2[s]/\langle s^2+s+1\rangle$ be the Alexander biquandle of Example~\ref{alex}. Consider an $X$-coloring $f:\mathcal{B}(L)\rightarrow X$. From the relations on $\mathcal{B}(L)$, we obtain,
	\begin{align*}
		f(x)&=f(z\ovr w)=f(z)\ovr f(w)=(s+1)f(z),\\
		f(y)&=f(w\un z)=f(w)\un f(z)=sf(w)+f(z),\\
		f(y)&=f(x\un z)=f(x)\un f(z)=sf(x)+f(z),\\
		f(w)&=f(z\ovr x)=f(z)\ovr f(x)=(s+1)f(z).
	\end{align*}
	Solving this system of linear equations yields,
	\[f(x)=f(w)=(s+1)f(z),\quad f(y)=0.\] Hence, if we denote an $X$-coloring $f$ of $L$ by $[f(x),f(y),f(w),f(z)]$, then the set of $X$-colorings of $L$ is
	\[\text{Col}_X(L)=\{f_0=[0,0,0,0],\;f_1=[1,0,1,s],\;f_2=[s,0,s,s+1],\;f_3=[s+1,0,s+1,1]\}.\]
	The operations $\un$ and $\ovr$ on $X$ are linear. Thus, End$(X)$ consists of the maps \[\sigma_0(x)=0,\quad \sigma_1(x)=x,\quad \sigma_2(x)=sx,\quad \sigma_3(x)=(s+1)x,\]
	for all $x\in X$. Then, the  $X$-coloring quiver of $L$ is as shown in Figure~\ref{fullbcq}. The edges resulting from $\sigma_0$, $\sigma_1$, $\sigma_2$ and $\sigma_3$ are indicated by blue, black, orange and green colors, respectively.
\end{example}
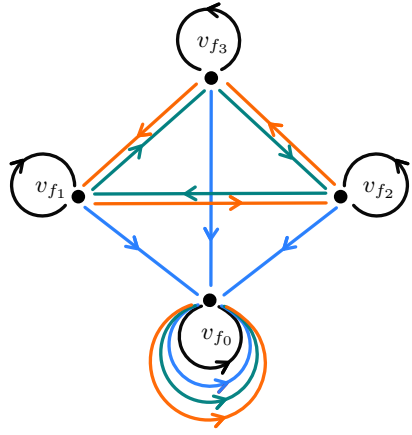
\begin{figure}
	\centering
	\begin{overpic}[width=0.35\linewidth]{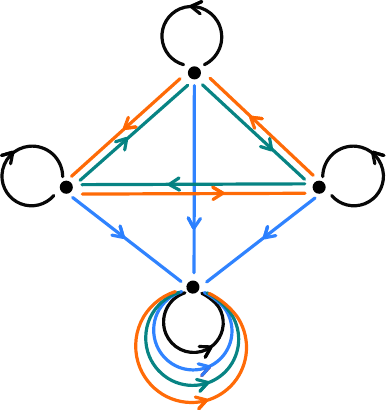}
		\put(45,21){\small$v_{f_0}$}
		\put(6,55.5){\small$v_{f_1}$}
		\put(83,55.5){\small$v_{f_2}$}
		\put(44,89.5){\small$v_{f_3}$}
	\end{overpic}
	\vspace{0.3 cm}
	\caption{An example of a biquandle coloring quiver.}
	\label{fullbcq}
\end{figure}
It is stated in \cite{cho2019quandle} that new invariants of links can be obtained by applying invariants of directed graphs to quandle coloring quivers, and this approach also extends to biquandle coloring quivers. As an immediate example, the \textit{in-degree $X$-coloring quiver polynomial} is introduced in the same paper using the following observation. Let $f$ be an $X$-coloring of a link diagram $L$. Consider a subset $S\subseteq \text{End}(X)$ and the corresponding $X$-coloring quiver $\mathcal{Q}_X^{S}(L)$. By definition, each endomorphism $\sigma\in S$ determines a directed edge $v_f\rightarrow v_{\sigma\circ f}$. It follows that the number of outgoing directed edges from each vertex $v$, denoted by deg$^-(v)$, is exactly $|S|$. However, the number of incoming directed edges to each vertex $v$, denoted by deg$^+(v)$, may vary.

\begin{definition}
	\cite{ceniceros2023psyquandle}
	Let $L$ be a link diagram, $X$ a finite biquandle, $S$ a subset of End$(X)$, and consider the corresponding $X$-coloring quiver $\mathcal{Q}_X^S(L)$. The \textit{in-degree $X$-coloring quiver polynomial} of $L$ is
	\[\Phi_X^{\text{deg}^+,S}(L)=\sum_{f\in \text{Col}_X(L)}u^{\text{deg}^+(v_f)},\]
	where $u$ is a formal variable.
\end{definition}
\begin{example}
	The in-degree $X$-coloring quiver polynomial of $L$ of Example~\ref{indeg} is
	\[\Phi_X^{\text{deg}^+,S}(L)=3u^3+u^7.\]
\end{example}

\subsection{Directed Clique Homology}
\label{homologies}

Directed clique complex is introduced in \cite{masulli2016topology} for finite directed graphs without loops and multi-edges. Its construction uses \textit{directed cliques} of a given directed graph considered as abstract simplices.

\begin{definition}
	An \textit{abstract ordered simplicial complex} $\mathcal{K}$ on a set $S$ is a collection of finite lists $(s_0,...,s_n)$, where $s_i\in S$ and $s_i\neq s_j$ for all $i\neq j$, such that
	\begin{itemize}
		\item for each $s\in S$, $(s)\in \mathcal{K}$,
		\item if $(s_0,...,s_n)\in \mathcal{K}$ and $(s_{i_0},...,s_{i_k})$ is a sublist of $(s_0,...s_n)$, then $(s_{i_0},...,s_{i_k})\in \mathcal{K}$.
	\end{itemize}
	An element $(s_0,...,s_n)\in \mathcal{K}$ is called an \textit{abstract $n$-simplex}, and the set of abstract $n$-simplices in $\mathcal{K}$ is denoted by $K_n$. Any sublist $(s_{i_0},...,s_{i_k})$ of $(s_0,...,s_n)$ is called a \textit{face} of the $n$-simplex $(s_0,...,s_n)$. A \textit{subcomplex} of $\mathcal{K}$ is a subcollection $\mathcal{L}\subseteq \mathcal{K}$ which is itself an abstract simplicial complex.
\end{definition}

Abstract $n$-simplices can be realized in $\mathbb{R}^{n+1}$ as geometric $n$-simplices. A \textit{geometric $n$-simplex} is the convex hull of the points $e_0,...,e_n\in \mathbb{R}^{n+1}$, where $e_i$ is the vector whose only nonzero entry is 1 in the $i$-th position. The \textit{geometric realization} of an abstract simplicial complex $\mathcal{K}$, denoted by $|\mathcal{K}|$, is the topological space obtained by replacing each abstract $n$-simplex in $\mathcal{K}$ by a geometric $n$-simplex and gluing the geometric $n$-simplices along their shared faces.

\begin{definition}
	\label{chain}
	Let $\mathcal{K}$ be an abstract ordered simplicial complex. For each $n\geq 0$, let $C_n(\mathcal{K})$ be the free abelian group generated by $K_n$, and define $\partial_n: C_n(\mathcal{K})\rightarrow C_{n-1}(\mathcal{K})$ on the generators by\[\partial_n(s_0,...,s_n)=\sum\limits_{i=0}^n(-1)^i(s_0,...,s_{i-1},s_{i+1},...,s_n)\]
	and extend linearly. Then,
	\begin{align*}
		\cdot\cdot\cdot \xrightarrow{\partial_{n+1}} C_n(\mathcal{K})\xrightarrow{\partial_{n}} C_{n-1}(\mathcal{K})\xrightarrow{\partial_{n-1}}\cdot\cdot\cdot\xrightarrow{\partial_2}C_1(\mathcal{K})\xrightarrow{\partial_1}C_0(\mathcal{K})\xrightarrow{}0
	\end{align*} satisfies the condition $\partial_{n+1}\circ \partial_n=0$, for all $n$ \cite{hatcher2002algebraic}. Hence, it is a chain complex called the \textit{simplicial chain complex} of $\mathcal{K}$, and denoted by $C_{\ast}(\mathcal{K})$.
\end{definition}

\begin{definition}
	\cite{hatcher2002algebraic}
	The \textit{$n$-th simplicial homology group} $H_n(\mathcal{K})$ of an abstract ordered simplicial complex $\mathcal{K}$ is defined as
	\[
	H_n(\mathcal{K})=\operatorname{ker}(\partial_n)/\operatorname{im}(\partial_{n+1}).
	\]
\end{definition}

\begin{definition}
	A \textit{directed graph} $G$ is a pair of finite sets $(V,E)$, where $V$ is called the \textit{vertex set} and $E\subseteq \{(u,v)\in V\times V\;|\; u\neq v\}$ is called the \textit{directed edge set} of $G$. A directed edge is denoted by the ordered pair $(u,v)$ or $u\rightarrow v$. A \textit{directed $(n+1)$-clique} in $G$ is a subset $W=\{v_0,...,v_n\}\subseteq V$ together with an ordering $(v_0,...,v_n)$ such that for all $i<j$, $v_i\rightarrow v_j\in E$.
\end{definition}
Given a directed graph $G=(V,E)$, one can construct an abstract ordered simplicial complex on $V$ using directed cliques of $G$ as follows.

\begin{definition}
	\label{directdef}
	Let $G=(V,E)$ be a finite directed graph. The \textit{directed clique complex} of $G$, denoted by $\mathcal{K}(G)$, is the abstract ordered simplicial complex defined by
	\begin{align*}
		K_0(G)&=\{(v)\;|\; v\in V\},\\
		K_n(G)&=\{(v_0,...,v_n)\;|\; v_i\rightarrow v_j\in E \text{ for all }i<j\} \quad\text{for } n>0.
	\end{align*}
\end{definition}
That is, the 0-simplex set $K_0(G)$ consists of the vertices of $G$, and for $n>0$, the $n$-simplex set $K_n(G)$ consists of directed $(n+1)$-cliques $(v_0,...,v_n)$ in $G$. Note that the ordering of an abstract $n$-simplex is determined by the directions of the edges in the corresponding directed $(n+1)$-clique. In Table~\ref{table1}, we illustrate examples of directed graphs on the leftmost column, the abstract ordered simplicial complexes associated with these graphs in the middle column, and the geometric realizations of these complexes on the rightmost column.

\begin{table}[h]
	\centering
	\setlength{\tabcolsep}{3pt}
	\renewcommand{\arraystretch}{0.8}
	
	\begin{tabular}{|>{\centering\arraybackslash}m{0.2\textwidth}|
			>{\centering\arraybackslash}m{0.45\textwidth}|
			>{\centering\arraybackslash}m{0.25\textwidth}|}
		\hline
		\vspace{0.1 cm}\text{\small Directed graph} & \vspace{0.1 cm}\text{\small Abstract ordered simplicial complex} & \vspace{0.1 cm}\text{\small Geometric realization}
		\tabularnewline
		\hline
		
		\begin{overpic}[width=0.8\linewidth]{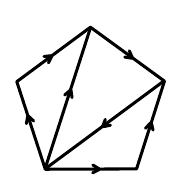}
			\put(9,5){$v_0$}
			\put(-5,55){$v_1$}
			\put(45,90){$v_2$}
			\put(90,55){$v_3$}
			\put(75,5){$v_4$}
		\end{overpic}
		&
		
		\scalebox{0.9}{$\begin{aligned}
				K_0 &= \{(v_0),(v_1),(v_2),(v_3),(v_4)\},\\
				K_1 &= \{(v_0,v_1),(v_0,v_2),(v_1,v_2),(v_2,v_3),\\
				&\qquad (v_3,v_0),(v_0,v_4),(v_4,v_3)\},\\
				K_2 &= \{(v_0,v_1,v_2)\}.
			\end{aligned}$}
		\vspace{0.1 cm}
		&
		\includegraphics[width=0.55\linewidth]{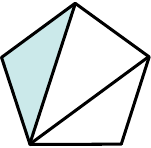}
		\tabularnewline
		\hline
		
		\begin{overpic}[width=0.8\linewidth]{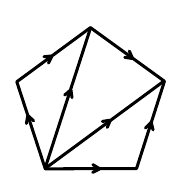}
			\put(9,5){$v_0$}
			\put(-5,55){$v_1$}
			\put(45,90){$v_2$}
			\put(90,55){$v_3$}
			\put(75,5){$v_4$}
		\end{overpic}
		&
		\scalebox{0.9}{$\begin{aligned}
				\small K_0 &= \{(v_0),(v_1),(v_2),(v_3),(v_4)\},\\
				K_1 &= \{(v_0,v_1),(v_0,v_2),(v_1,v_2),(v_2,v_3),\\&\qquad(v_0,v_3),(v_0,v_4),(v_4,v_3)\},\\
				K_2 &= \{(v_0,v_1,v_2), (v_0,v_2,v_3),(v_0,v_4,v_3)\}.
			\end{aligned}$}
		\vspace{0.1 cm}
		&
		\includegraphics[width=0.55\linewidth]{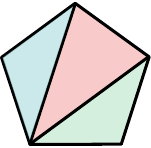}
		\tabularnewline
		\hline
		
		\begin{overpic}[width=0.8\linewidth]{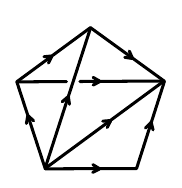}
			\put(9,5){$v_0$}
			\put(-5,55){$v_1$}
			\put(45,90){$v_2$}
			\put(90,55){$v_3$}
			\put(75,5){$v_4$}
		\end{overpic}
		&
		\vspace{0.1 cm}
		\scalebox{0.9}{$\begin{aligned}
				K_0 &= \{(v_0),(v_1),(v_2),(v_3),(v_4)\},\\
				K_1 &= \{(v_0,v_1),(v_0,v_2),(v_1,v_2),(v_1,v_3),\\&\qquad(v_2,v_3),(v_0,v_3),(v_0,v_4),(v_4,v_3)\},\\
				K_2 &= \{(v_0,v_1,v_2), (v_0,v_2,v_3),(v_0,v_4,v_3)\\&\qquad(v_1,v_2,v_3),(v_0,v_1,v_3)\},\\K_3&=\{(v_0,v_1,v_2,v_3)\}.
			\end{aligned}$}
		\vspace{0.1 cm}
		&
		\includegraphics[width=0.55\linewidth]{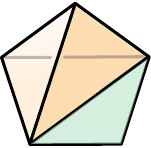}
		\tabularnewline
		\hline
	\end{tabular}
	\vspace{0.3 cm}
	\caption{Examples of directed graphs, the associated abstract ordered simplicial complexes, and their geometric realizations. Colored regions in a geometric realization indicate simplices determined by directed cliques in the corresponding graph.}
	\label{table1}
\end{table}
\begin{definition}
	Let $G=(V,E)$ be a finite directed graph without loops or multi-edges. The \textit{$n$-th directed clique homology group} of $G$, denoted by $H_n(G)$, is defined to be the $n$-th simplicial homology group of $\mathcal{K}(G)$. That is,
	\[H_n(G):= H_n(\mathcal{K}(G)).\]
\end{definition}

\begin{remark}
	In \cite{masulli2016topology}, the construction is carried out over the field $\mathbb{Z}_2$. That is, $C_n(\mathcal{K})$ in Definition~\ref{chain} is considered a vector space over $\mathbb{Z}_2$ with basis $K_n$ rather than a free abelian group generated by $K_n$. Moreover, the authors of \cite{masulli2016topology} are interested in the dimension of the homology group $H_n(\mathcal{K})$, which is \[\text{dim}(H_n(\mathcal{K}))=\text{dim}(\text{ker}(\partial_n))-\text{dim}(\text{im}(\partial_{n+1})),\] rather than the generators themselves. However, in this paper, the generators of $H_n(\mathcal{\mathcal{K}})$ themselves play a central role in the construction of invariants of links introduced in Section~\ref{mainsection1}.
\end{remark}

\subsection{Persistent Homology}
\label{persistenthom}

We begin with the notion of a filtration, which we will use in several settings throughout this paper.

\begin{definition}
	Let $\mathcal{C}$ be a class of mathematical objects equipped with a notion of inclusion. A \textit{filtration} in $\mathcal{C}$ is a nested sequence
	\[
	A_0 \subseteq A_1 \subseteq \cdots \subseteq A_m,
	\]
	where $A_i\in \mathcal{C}$ for all $i=0,...,m$.
	In particular, if $\mathcal{K}$ is an abstract ordered simplicial complex, a filtration of $\mathcal{K}$ is a nested sequence of its subcomplexes
	\[
	\mathcal{K}_0 \subseteq \mathcal{K}_1 \subseteq \cdots \subseteq \mathcal{K}_m = \mathcal{K}.
	\]
\end{definition}

\begin{definition}
	Two filtrations
	\[
	A_0 \subseteq A_1 \subseteq \cdots \subseteq A_m
	\quad \text{and} \quad
	B_0 \subseteq B_1 \subseteq \cdots \subseteq B_m
	\]
	are said to be \textit{isomorphic} if there exist isomorphisms
	\;$
	\varphi_i : A_i \to B_i
	$\;
	for each $i$ such that the following diagram commutes:
	\[
	\begin{array}{ccc}
		A_i\;\; & \hookrightarrow & A_{i+1}\;\; \\
		\downarrow \text{\footnotesize$\varphi_i$} &  & \downarrow \text{\footnotesize$\varphi_{i+1}$} \\
		B_i\;\; & \hookrightarrow & B_{i+1}\;\;
	\end{array}
	\]
	where inclusion and isomorphism are understood in the corresponding setting.
\end{definition}

In this paper, we apply this isomorphism notion to filtrations of quivers, abstract ordered simplicial complexes and their corresponding chain complexes.

Consider a filtration $\{\mathcal{K}_i\}_{i=0,...,m}$ of an abstract ordered simplicial complex $\mathcal{K}$. For each $i\leq j$, the inclusion $\mathcal{K}_i\hookrightarrow \mathcal{K}_{j}$ induces an inclusion $C_n(\mathcal{K}_i)\hookrightarrow C_n(\mathcal{K}_{j})$ for all $n\geq 0$, and hence induces an inclusion chain map $C_{\ast}(\mathcal{K}_i)\hookrightarrow C_{\ast}(\mathcal{K}_{j})$. This yields a filtration of the chain complex $C_{\ast}(\mathcal{K})$, that is, a nested sequence of chain subcomplexes 
\begin{align}
	C_{\ast}(\mathcal{K}_0)\subseteq C_{\ast}(\mathcal{K}_1)\subseteq...\subseteq C_{\ast}(\mathcal{K}_m)=C_{\ast}(\mathcal{K}).
	\label{chainfilt}
\end{align}
Applying the homology functor $H_n$ to Filtration \eqref{chainfilt}, we obtain 
\begin{align*}
	H_n(\mathcal{K}_0)\rightarrow H_n(\mathcal{K}_1) \rightarrow ... \rightarrow H_n(\mathcal{K}_m)=H_n(\mathcal{K}).
\end{align*}

\begin{definition}
	\cite{edelsbrunner2002topological}
	Let $\mathcal{K}_0\subseteq \mathcal{K}_1\subseteq ... \subseteq \mathcal{K}_m$ be a filtration of an abstract ordered simplicial complex $\mathcal{K}$. For $i\leq j$, the \textit{n-th persistent homology group} of $\mathcal{K}$ is defined by
	\begin{align*}
		H_{n}^{i,j}(\mathcal{K})=\text{im}(H_n(\mathcal{K}_i)\rightarrow H_n(\mathcal{K}_j)).
	\end{align*}
	In other words, for all $i\leq j$, the $n$-th persistent homology group consists of homology classes that were \textit{born} at or before the $i$-th stage and are still \textit{alive} in the $j$-th stage in the filtration.
\end{definition}

Persistence barcodes are used to keep track of homological features throughout a filtration. Each homology class has a birth index $i$. Some of these classes have a death index $j$, while others may persist throughout the filtration, in which case we set $j = \infty$. For each homology class, the birth and death indices together determine an interval $[i,j)$ called a \textit{persistence interval}. The multiset of persistence intervals over all homology classes is called the \textit{persistence barcode.} 
Persistence barcodes are usually visualized in the $xy$-plane where each persistence interval is represented by a horizontal bar, with indices on the $x$-axis.

A geometric simplicial complex, together with a choice of orientation for each of its simplices, naturally determines an abstract ordered simplicial complex. This abstract ordered simplicial complex is obtained by forming abstract ordered simplices using the $0$-simplices in the geometrix simplicial complex. The order on these abstract simplices is with respect to the orientation of the corresponding geometric simplex. In the geometric setting, simplicial homology has an intuitive interpretation: the 0-th homology group counts the number of connected components and for $n\geq 1$, the $n$-th homology group detects the number of $n$-dimensional holes in the simplicial complex. In Example~\ref{ex}, we omit the computations and utilize this interpretation.
\begin{figure}
	\centering
	\begin{overpic}[width=0.52\linewidth]{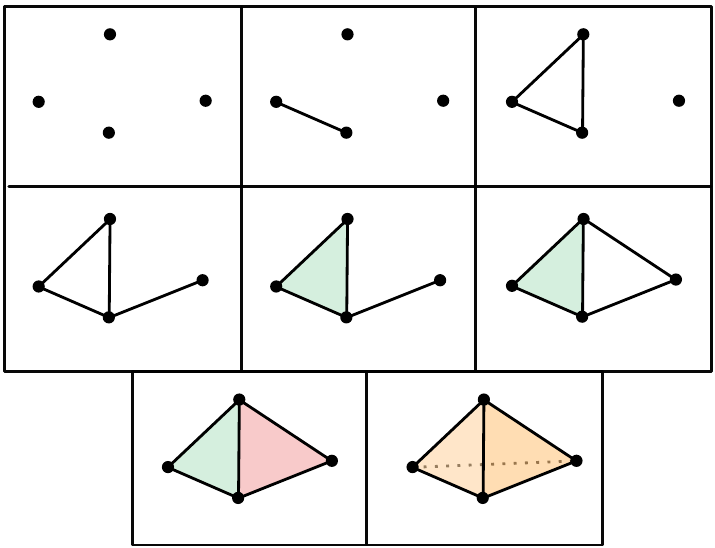}
		\put(3,53){\small$\mathcal{K}_0$}
		\put(36,53){\small$\mathcal{K}_1$}
		\put(69,53){\small$\mathcal{K}_2$}
		\put(3,27.5){\small$\mathcal{K}_3$}
		\put(36,27.5){\small$\mathcal{K}_4$}
		\put(69,27.5){\small$\mathcal{K}_5$}
		\put(21,3){\small$\mathcal{K}_6$}
		\put(53,3){\small$\mathcal{K}_7=\mathcal{K}$}
	\end{overpic}
	\vspace{0.3 cm}
	\caption{A filtration of a geometric simplicial complex.}
	\label{pcd}
\end{figure}
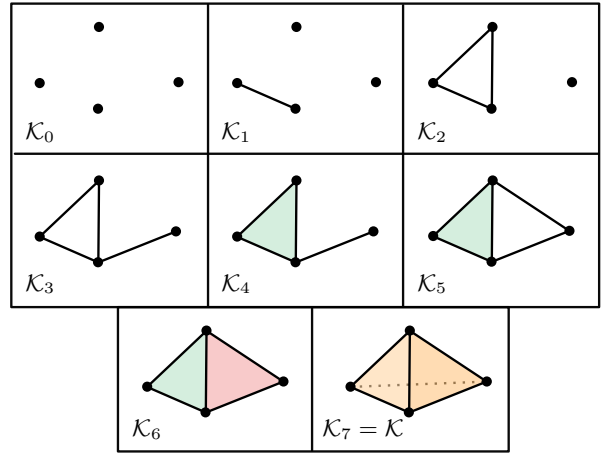

\begin{example}
	\label{ex}
	A filtration of a geometric simplicial complex $\mathcal{K}$ is given in Figure~\ref{pcd}. Observe that the number of connected components in $\mathcal{K}_0$ is four and it decreases at each step and finally, in $\mathcal{K}_3$, there is only one connected component. A 1-dimensional hole appears in $\mathcal{K}_2$ and persists through $\mathcal{K}_3$; however in $\mathcal{K}_4$, it is filled. Another 1-dimensional hole is born in $\mathcal{K}_5$, but it dies at $\mathcal{K}_6$. Under these observations, the persistence barcode corresponding to the given filtration of $\mathcal{K}$ is given by
	\[\{[0,\infty),[0,1),[0,2),[0,3),[2,4),[5,6)\},\]
	where a generator which is born at step $i$ and survives throughout the filtration is denoted by the interval $[i,\infty)$.
	The persistence barcode is visualized in Figure~\ref{barcode}.
	
\end{example}

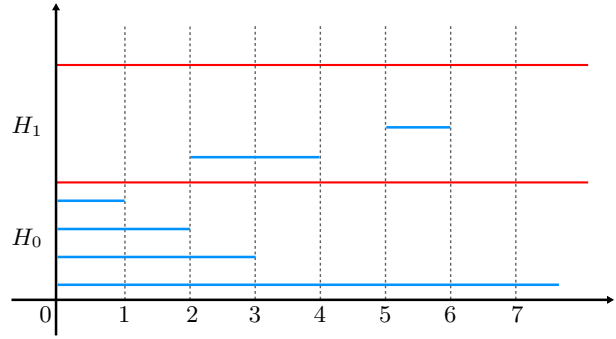
\begin{figure}
	\centering
	\begin{overpic}[scale=0.7]{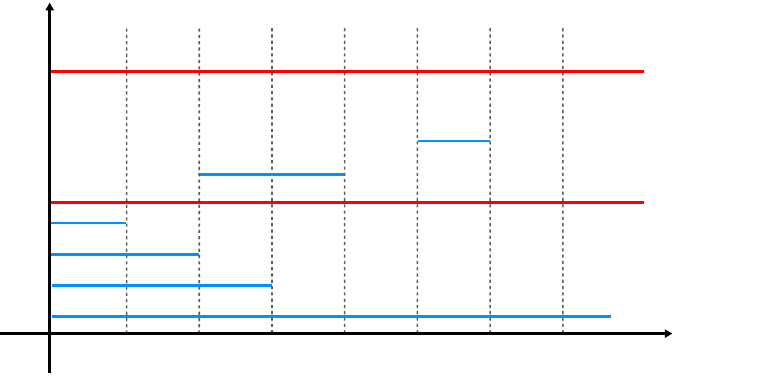}
		\put(0,14){\small $H_0$}
		\put(0,30){\small $H_1$}
		\put(4,3){\small $0$}
		\put(15.3,3){\small $1$}
		\put(25,3){\small $2$}
		\put(34,3){\small $3$}
		\put(43.3,3){\small $4$}
		\put(52.7,3){\small $5$}
		\put(62,3){\small $6$}
		\put(71.6,3){\small $7$}
	\end{overpic}
	\vspace{0.3 cm}
	\caption[Barcode]{A persistence barcode of a simplicial complex $\mathcal{K}$.}
	\label{barcode}
\end{figure}

Persistent homology depends on the choice of filtration and is therefore not an invariant of the underlying object in general. In many applications, particularly in topological data analysis, filtrations arise from geometric or data-driven constructions and are therefore sensitive to deformations of the underlying object, even when these deformations preserve its equivalence class. In such settings, the short persistence intervals are often considered as noise and ignored, while longer intervals are regarded as meaningful data revealing the topological features of the underlying object. We refer the interested reader to \cite{edelsbrunner2008persistent} for motivation, historical background, and further details on persistent homology.

\section{Persistent Homology of Biquandle Coloring Quivers}
\label{mainsection1}

\subsection{$\boldsymbol{N}$-directed Clique Homology}
\label{ndirectsec}
In this section, we generalize the directed clique homology of directed graphs to quivers. In \cite{masulli2016topology}, directed graphs are not allowed to have loops or multiple edges. However, in biquandle coloring quivers, loops and multiple edges arise naturally.

\begin{definition}
	\label{ndirectdef}
	Let $\mathcal{Q}$ be a finite quiver with vertex set $V$ and edge set $E$. Let $e_\mathcal{Q}(v_i,v_j)$ denote the number of directed edges from $v_i$ to $v_j$ in $\mathcal{Q}$, and let $N\geq 1$ be a fixed integer. The \textit{$N$-directed clique complex} $\mathcal{K}^{(N)}(\mathcal{Q})$ of $\mathcal{Q}$ is the abstract ordered simplicial complex defined as
	\begin{align*}
		K_0^{(N)}(\mathcal{Q})&=\{(v)\;|\; v\in V\},\notag\\
		K_n^{(N)}(\mathcal{Q})&=\{(v_0,...,v_n)\;|\; v_i\neq v_j \text{ and } e_\mathcal{Q}(v_i,v_j)\geq N, \text{ for all }i<j \}, \text{ for }n>0.
	\end{align*} 
\end{definition}

Observe that the 1-directed clique complex coincides with the directed clique complex of \cite{masulli2016topology} given in Section~\ref{homologies}. The existence of loops in quivers may result in degenerate lists in which vertices are repeated. We eliminate such lists by the condition $v_i\neq v_j$ for all $i\neq j$. Moreover, we modify the connectivity condition in the definition of a clique: instead of requiring exactly one directed edge, we require at least $N$ directed edges between any two ordered distinct vertices in the subgraph. Our motivation for such a modification is that we would like the existence of multi-edges in biquandle coloring quivers to play a role in the construction of invariants introduced in this paper. 

\begin{remark}
	\label{justremark}
	One could also generalize the approach in \cite{masulli2016topology} by setting $e_G(v_i,v_j)=N$ in Definition~\ref{ndirectdef}. It is clear that this choice yields fewer simplices in the associated $N$-directed clique complex compared to our choice $e_G(v_i,v_j)\geq N$ in general. The reason for our choice is justified in Section~\ref{persistsec}.
\end{remark}

\begin{definition}
	\label{homdef1}
	The \textit{$n$-th $N$-directed clique homology group} of a quiver $\mathcal{Q}$ is defined by
	\[H_n^{(N)}(\mathcal{Q}):=H_n(\mathcal{K}^{(N)}(\mathcal{Q})).\]
\end{definition}

\begin{definition}
	Two quivers $\mathcal{Q}_1=(V_1,E_1)$ and $\mathcal{Q}_2=(V_2,E_2)$ are \textit{isomorphic}, if there is a bijection $\,\psi:V_1\rightarrow V_2\,$
	such that $\,e_{\mathcal{Q}_1}(u,v)=e_{\mathcal{Q}_2}(\psi(u),\psi(v))\,$ for all $\,u,v\in V_1$.
\end{definition}

\begin{proposition}
	\label{quivinv2}
	Let $\mathcal{Q}_1$ and $\mathcal{Q}_2$ be two finite quivers. If $\mathcal{Q}_1$ and $\mathcal{Q}_2$ are isomorphic, then their $N$-directed clique complexes $\mathcal{K}^{(N)}(\mathcal{Q}_1)$ and $\mathcal{K}^{(N)}(\mathcal{Q}_2)$ are isomorphic as abstract simplicial complexes. As a consequence, 
	\[H_n^{(N)}(\mathcal{Q}_1)\cong H_n^{(N)}(\mathcal{Q}_2)\]
	for all $n\geq 0$ and $N\geq 1$.
\end{proposition}

\begin{proof}
	Since $\mathcal{Q}_1$ and $\mathcal{Q}_2$ are isomorphic, there is a bijection $\psi:V_1\rightarrow V_2$ where $V_i$ is the vertex set of $\mathcal{Q}_i$, for $i=1,2$. It is clear that $(v)\in K_0^{(N)}(\mathcal{Q}_1)$ if and only if $\psi(v)\in K_0^{(N)}(\mathcal{Q}_2)$. For $n\geq 1$, since $\psi$ preserves the number of directed edges between any pair of vertices, we have that $(v_0,...,v_n)\in K_n^{(N)}(\mathcal{Q}_1)$ if and only if $(\psi(v_0),...,\psi(v_n))\in K_n^{(N)}(\mathcal{Q}_2)$. Hence, the bijection $\psi:V_1\rightarrow V_2$ induces a bijection between the set of abstract $n$-simplices $K_n^{(N)}(\mathcal{Q}_1)$ and $K_n^{(N)}(\mathcal{Q}_2)$ for all $n$. Therefore, the induced maps on chain complexes of $\mathcal{K}^{(N)}(\mathcal{Q}_1)$ and $\mathcal{K}^{(N)}(\mathcal{Q}_2)$ are isomorphisms compatible with the boundary maps, and hence induce isomorphisms between homology groups.
\end{proof}

\begin{remark}
	\label{func}
	The construction $\mathcal Q \mapsto \mathcal K^{(N)}(\mathcal Q)$ admits a functorial
	interpretation in a categorical setting of quivers with suitable morphisms, from which the isomorphisms of $\mathcal K^{(N)}(\mathcal Q_1)$ and $\mathcal K^{(N)}(\mathcal Q_2)$ in the proof of Theorem~\ref{quivinv2} could also be obtained.
	Consequently, by functoriality of simplicial homology, the assignment
	\[
	\mathcal Q \mapsto H_n^{(N)}(\mathcal Q)
	:= H_n(\mathcal K^{(N)}(\mathcal Q))
	\]
	is functorial on quiver isomorphisms. This observation will be used later when applying
	$H_n^{(N)}$ to isomorphisms of quiver filtrations.
\end{remark}

\begin{definition}
	\label{homdef}
	Let $L$ be a link diagram, $X$ be a finite biquandle, $S\subseteq$ End$(X)$, and $N\geq 1$. The \textit{$n$-th $N$-directed clique homology group} of $L$ with respect to $X$ and $S$ is defined by
	\[
	H_n^{(N)}(L;X,S):=H_n^{(N)}(\mathcal{Q}_X^S(L)).
	\]
\end{definition}

\begin{theorem}
	Let $X$ be a finite biquandle, $S\subseteq$ End$(X)$. Then, for all $n\geq 0$ and $N\geq1$, the $n$-th $N$-directed clique homology group $H_n^{(N)}(L;X,S)$ is an invariant of links.
\end{theorem}

\begin{proof}
	Let $L_1$ and $L_2$ be two virtually equivalent link diagrams. By Theorem~\ref{quivinv}, the $X$-coloring quivers $\mathcal{Q}_X^{S}(L_1)$ and $\mathcal{Q}_X^{S}(L_2)$ are isomorphic. By Proposition~\ref{quivinv2},
	\[H_n^{(N)}(\mathcal{Q}_X^S(L_1))\cong H_n^{(N)}(\mathcal{Q}_X^S(L_2)).\] Hence, by definition,
	\[
	H_n^{(N)}(L_1;X,S)\cong H_n^{(N)}(L_2;X,S).
	\]
\end{proof}
\begin{example}
	Consider the virtual knot diagram $L$ given in Figure \ref{lvkd} and the biquandle $\mathbb{Z}_3$ with the operation matrix 
	\[\left[\begin{array}{c c c | c c c}
		1& 1& 1& 1& 1& 1\\
		3& 2& 2& 3& 2& 2\\
		2& 3& 3& 2& 3& 3
	\end{array}\right].\]
	Considering the relations in Figure~\ref{crossingrels} together with the operation matrix of $\mathbb{Z}_3$, one observes that $L$ admits only trivial $\mathbb{Z}_3$-colorings. That is, $f:\mathcal{B}(L)\rightarrow \mathbb{Z}_3$ is a $\mathbb{Z}_3$-coloring of $L$ if and only if $f$ is a constant map. Hence, $L$ has three $\mathbb{Z}_3$-colorings given by
	\[\text{Col}_{\mathbb{Z}_3}(L)=\{f_0
	=[1,1,1,1],f_1=[2,2,2,2],f_2=[3,3,3,3]\}.\]
	A direct computation shows that the endomorphism set of $\mathbb{Z}_3$ with the given operation matrix is,
	\[\text{End}(\mathbb{Z}_3)=\{[1,1,1],[1,2,3],[1,3,2],[2,2,2],[2,3,3],[3,2,2],[3,3,3]\}.\]
	We set $S=\text{End}(X)$ and obtain the full $\mathbb{Z}_3$-coloring quiver of $L$ depicted in Figure \ref{exN}, where an arrow $v_{f_i}\rightarrow v_{f_j}$ labeled by $n$ indicates that $e(v_{f_i},v_{f_j})=n$.
	
	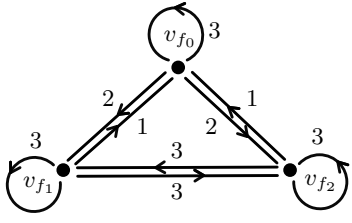
\begin{figure}
		\centering
		\begin{overpic}[scale=0.7]{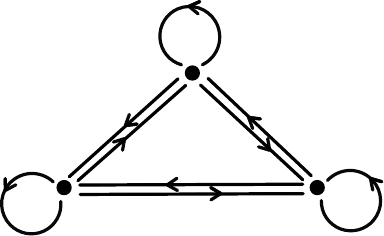}
			\put(46,50){\small $v_{f_0}$}
			\put(5,7.5){\small $v_{f_1}$}
			\put(87,8){\small $v_{f_2}$}
			
			\put(38,23){\small $1$}
			\put(70,31){\small $1$}
			\put(48,16){\small $3$}
			\put(28,31){\small $2$}
			\put(58,23){\small $2$}
			\put(48,4){\small $3$}
			\put(7,19){\small $3$}
			\put(89,20){\small $3$}
			\put(59,51){\small $3$}
		\end{overpic}
		\vspace{0.3 cm}
		\caption[Barcode]{An example of a biquandle coloring quiver.}
		\label{exN}
	\end{figure}
	
	The $N$-directed clique complex $\mathcal{K}^{(N)}(\mathcal{Q}_{\mathbb{Z}_3}(L))$ and the $n$-th $N$-directed clique homology group $H_n^{(N)}(L;\mathbb{Z}_3,S)\;$ of $\;L\;$ for $\;n=0,1\;$ and $\;N=1,2,3,4\;$ are given in Table~\ref{table2}. Observe that $\mathcal{K}^{(4)}(\mathcal{Q}_{\mathbb{Z}_3}(L))=\mathcal{K}^{(N)}(\mathcal{Q}_{\mathbb{Z}_3}(L))$ and hence $H_n^{(4)}(L;\mathbb{Z}_3,S)\cong H_n^{(N)}(L;\mathbb{Z}_3,S)$ for all $n\geq 0$ and for $N\geq 4$.
	\begin{table}[h]
		\centering
		\setlength{\tabcolsep}{3pt}
		\renewcommand{\arraystretch}{0.8}
		
		\begin{tabular}{|>{\centering\arraybackslash}m{0.1\textwidth}|
				>{\centering\arraybackslash}m{0.62\textwidth}|
				>{\centering\arraybackslash}m{0.2\textwidth}|}
			\hline
			& \vspace{0.1 cm}$\mathcal{K}^{(N)}(\mathcal{Q}_{\mathbb{Z}_3}(L))$ \vspace{0.05 cm}& \vspace{0.1 cm}$H_n^{(N)}(L;\mathbb{Z}_3,S)$\vspace{0.05 cm}
			\tabularnewline
			\hline
			
			$N=1$
			&
			\begin{minipage}{\linewidth}
				\centering
				\vspace{0.2 cm}
				\scalebox{0.9}{$\begin{aligned}
						K_0 &= \{(v_{f_0}),(v_{f_1}),(v_{f_2})\},\\
						K_1 &= \{(v_{f_0},v_{f_1}),(v_{f_1},v_{f_0}),(v_{f_0},v_{f_2}),(v_{f_2},v_{f_0}),(v_{f_1},v_{f_2}),(v_{f_2},v_{f_1})\},\\\
						K_2 &= \{(v_{f_0},v_{f_1},v_{f_2}), (v_{f_0},v_{f_2},v_{f_1}),(v_{f_1},v_{f_0},v_{f_2})\\&\qquad(v_{f_1},v_{f_2},v_{f_0}),(v_{f_2},v_{f_0},v_{f_1}),(v_{f_2},v_{f_1},v_{f_0})\}.
					\end{aligned}$}
				\vspace{0.2 cm}
			\end{minipage}
			&
			\scalebox{0.9}{$\begin{aligned}
					H_0 &\cong \mathbb{Z},\\
					H_1 & \cong \{0\},\\
					H_2 & \cong \mathbb{Z}^2,\\
					H_n & \cong \{0\},\; n\geq 3.
				\end{aligned}$}
			\tabularnewline
			\hline
			
			$N=2$
			&
			\vspace{0.2 cm}
			\scalebox{0.9}{$\begin{aligned}
					K_0 &= \{(v_{f_0}),(v_{f_1}),(v_{f_2})\},\\
					K_1 &= \{(v_{f_0},v_{f_1}),(v_{f_0},v_{f_2}),(v_{f_1},v_{f_2}),(v_{f_2},v_{f_1})\},\\
					K_2 &= \{(v_{f_0},v_{f_1},v_{f_2}), (v_{f_0},v_{f_2},v_{f_1})\}.
				\end{aligned}$}
			\vspace{0.2 cm}
			&
			\scalebox{0.9}{$\begin{aligned}
					H_0 &\cong \mathbb{Z},\\
					H_n & \cong \{0\},\; n\geq 1.
				\end{aligned}$}
			\tabularnewline
			\hline
			
			$N=3$
			&
			\vspace{0.2 cm}
			\scalebox{0.9}{$\begin{aligned}
					K_0 &= \{(v_{f_0}),(v_{f_1}),(v_{f_2})\},\\
					K_1 &= \{(v_{f_1},v_{f_2}),(v_{f_2},v_{f_1})\}.
				\end{aligned}$}
			\vspace{0.2 cm}
			&
			\vspace{-0.3 cm}
			\scalebox{0.9}{$\begin{aligned}
					H_0 &\cong \mathbb{Z}^2,\\
					H_1 &\cong \mathbb{Z},\\
					H_n & \cong \{0\},\; n\geq 2.
				\end{aligned}$}
			\tabularnewline
			\hline
			$N=4$
			&
			\vspace{0.2 cm}
			\scalebox{0.9}{$\begin{aligned}
					K_0 &= \{(v_{f_0}),(v_{f_1}),(v_{f_2})\}.
				\end{aligned}$}\vspace{0.2 cm}
			&
			\scalebox{0.9}{$\begin{aligned}
					H_0 &\cong \mathbb{Z}^3,\\
					H_n & \cong \{0\},\; n\geq 1.
				\end{aligned}$}
			\tabularnewline
			\hline
		\end{tabular}
		\vspace{0.3 cm}
		\caption{Examples of $N$-directed clique complexes and their $n$-th homology groups. $K_n=\emptyset$ whenever it does not appear in the middle column.}
		\label{table2}
	\end{table}
	
	The homology groups listed in Table~\ref{table2} can be computed using linear algebra. We omit these computations in this example, but mention the geometric interpretation of the homology groups. Observe that the geometric realization of the abstract $2$-simplices $(v_{f_i},v_{f_j},v_{f_k})$ and $(v_{f_i},v_{f_k},v_{f_j})$ is a disk whose boundary is formed by the edges $(v_{f_j},v_{f_k})$ and $(v_{f_k},v_{f_j})$. In the case $N=1$, there are three such pairs of $2$-simplices in $\mathcal{K}^{(1)}(\mathcal{Q}_{\mathbb{Z}_3}(L))$. Consequently, $\mathcal{K}^{(1)}(\mathcal{Q}_{\mathbb{Z}_3}(L))$ can be realized as three disks identified along their boundaries. This space has a single connected component, no 1-dimensional holes, and two independent 2-dimensional holes. Accordingly, 
	
	\[H_0^{(1)}(L;\mathbb{Z}_3,S)\cong \mathbb{Z},\quad H_1^{(1)}(L;\mathbb{Z}_3,S)\cong \{0\},\quad H_2^{(1)}(L;\mathbb{Z}_3,S)\cong \mathbb{Z}^2,\]
	as it is listed in Table~\ref{table2}. In the case $N=2$,\; $\mathcal{K}^{(2)}(\mathcal{Q}_{\mathbb{Z}_3}(L))$ can be realized as a disk, whose homology groups agree with those in Table~\ref{table2}. In the case $N=3$,\; $\mathcal{K}^{(3)}(\mathcal{Q}_{\mathbb{Z}_3}(L))$ can be realized as two connected components: a point corresponding to $v_{f_0}$, and a 1-dimensional hole formed by the edges  corresponding to $(v_{f_1},v_{f_2})$ and $(v_{f_2},v_{f_1})$. Hence, \[
	H_0^{(3)}(L;\mathbb{Z}_3,S)\cong \mathbb{Z}^2,\quad H_1^{(3)}(L;\mathbb{Z}_3,S)\cong \mathbb{Z}.\] Finally, in the case $N\geq 4$,\; $\mathcal{K}^{(N)}(\mathcal{Q}_{\mathbb{Z}_3}(L))$ consists of three distinct points. Hence, the only nontrivial homology group is $H_0^{(N)}(L;\mathbb{Z}_3,S)\cong \mathbb{Z}^3$.
	
\end{example}

\begin{figure}
	\centering
	\includegraphics[width=0.55\linewidth]{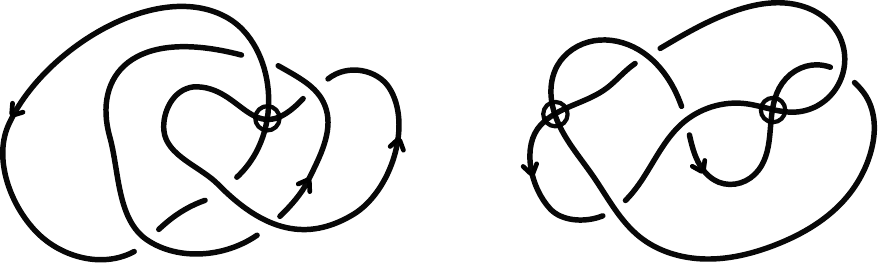}
	\vspace{0.3 cm}
	\caption{Two links $L_1$ (left-hand side) and $L_2$ (right-hand side) related to Example~\ref{ex1}}
	\label{vlinks}
\end{figure}

\begin{example}
	\label{ex1}
	Consider the link diagrams $L_1$ and $L_2$ in Figure~\ref{vlinks}, and the biquandle $\mathbb{Z}_4$ with the operation matrix
	\[\left[\begin{array}{c c c c | c c c c}
		1& 3& 1& 3& 1& 1& 1& 1\\
		4& 2& 4& 2& 2& 2& 2& 2\\
		3& 1& 3& 1& 3& 3& 3& 3\\
		2& 4& 2& 4& 4& 4& 4& 4
	\end{array}\right].\]
	Let $S=\{[3,2,1,4],[1,1,1,1],[3,1,3,1],[1,4,3,2],[3,4,1,2],[1,3,1,3],[3,3,3,3]\}\subset \text{End}(\mathbb{Z}_4)$ and set $N=2$. The number of $\mathbb{Z}_4$-colorings of $L_1$ and $L_2$ are equal: $\Phi_X^\mathbb{Z}(L_1)=16=\Phi_X^\mathbb{Z}(L_2)$. However, the $2$-directed clique homology groups of $L_1$ and $L_2$ with respect to $\mathbb{Z}_4$ and $S$ are
	\begin{align*}
		&H_0^{(2)}(L_1;\mathbb{Z}_4,S)\cong \mathbb{Z},&&H_0^{(2)}(L_2;\mathbb{Z}_4,S)\cong \mathbb{Z}^9,\\
		&H_1^{(2)}(L_1;\mathbb{Z}_4,S)\cong \{0\},&&H_1^{(2)}(L_2;\mathbb{Z}_4,S)\cong \{0\},\\
		&H_2^{(2)}(L_1;\mathbb{Z}_4,S)\cong \mathbb{Z}^6,&&H_2^{(2)}(L_2;\mathbb{Z}_4,S)\cong \mathbb{Z}^2,\\
		&H_3^{(2)}(L_1;\mathbb{Z}_4,S)\cong \mathbb{Z}^7,&&H_3^{(2)}(L_2;\mathbb{Z}_4,S)\cong \mathbb{Z}^3,
	\end{align*}
	and $H_n^{(2)}(L_i;\mathbb{Z}_4,S)\cong \{0\}$ for $i=1,2$ and $n\geq 4$. Although $L_1$ and $L_2$ have the same number of $\mathbb{Z}_4$-colorings, the 2-directed clique homology groups differ significantly. In particular, the ranks of $H_0^{(2)}$, $H_2^{(2)}$, and $H_3^{(2)}$ distinguish $L_1$ and $L_2$.
\end{example}

\subsection{Filtrations of Biquandle Coloring Quivers}
\label{filters}

In this section, we construct a particular type of filtrations of biquandle coloring quivers. Let $X$ be a finite biquandle and $L$ be a link diagram. We consider a nested sequence of subsets $\{S_i\}_{i=0,...,m}$ of End$(X)$ in the form
\[S_0\subseteq S_1\subseteq ... \subseteq S_m.\]
This sequence forms a filtration of the endomorphism set of the biquandle $X$. We denote this filtration by $S_\ast$. Each subset $S_i\subset \text{End}(X)$ corresponds to a unique $X$-coloring quiver $\mathcal{Q}_X^{S_i}(L)$ of $L$ up to quiver isomorphism for $i=0,...,m$. Notice that for any two choices $S_i$ and $S_j$ such that $S_i\subseteq S_j$, the $X$-coloring quiver $\mathcal{Q}_X^{S_i}(L)$ is naturally a subquiver of the $X$-coloring quiver $\mathcal{Q}_X^{S_j}(L)$ with the same vertex set. We use the notation $\mathcal{Q}_X^{S_i}(L)\subseteq \mathcal{Q}_X^{S_j}(L)$ for this relation. 

\begin{definition}
	Let $L$ be a link diagram, $X$ a finite biquandle, and consider a filtration $S_\ast=\{S_i\}_{i=0,...,m}$ of End$(X)$. The \textit{$X$-coloring quiver filtration} of $L$, denoted by $\mathcal{Q}_X^{S_\ast}(L)$, is defined as the filtration 
	\[\mathcal{Q}_X^{S_0}(L)\subseteq \mathcal{Q}_X^{S_1}(L)\subseteq...\subseteq \mathcal{Q}_X^{S_m}(L)\]
	of the full $X$-coloring quiver $\mathcal{Q}_X(L)$.
\end{definition}

\begin{remark}
	In many classical settings, a filtration of an object is assumed to terminate at the ambient object. In our setting, biquandle coloring quivers are used to distinguish links. This can be achieved using filtrations that do not necessarily reach the full biquandle coloring quiver. Allowing such filtrations provides flexibility and can be advantageous computationally in many cases. For instance, it is not an easy task to determine the full endomorphism set for large biquandles, and the number of edges in the full biquandle coloring quiver can be quite large when the endomorphism set is large. Hence, we do not require our filtrations to terminate at the full biquandle coloring quiver. That is, we do not require $S_m=\text{End}(X)$ in a filtration $S_\ast=\{S_i\}_{i=0,...,m}$ of End$(X)$.
\end{remark}

\begin{example}
	\label{filtex}
	Let $X$ be the Alexander biquandle $Z_5$ with $t=2$ and $r=4$.
	It can be verified that the endomorphism set of $\mathbb{Z}_5$ with the given biquandle structure is
	\begin{align*}
		\text{End}(\mathbb{Z}_5)=\{&\sigma_1=[1,2,3,4,5],
		\sigma_2=[2,4,1,3,5],
		\sigma_3=[3,1,4,2,5],\\
		&\sigma_4=[4,3,2,1,5],
		\sigma_5=[5,5,5,5,5]\}.
	\end{align*}
	
	\noindent Consider the filtration
	\[S_\ast:\quad S_0\subseteq S_1\subseteq ... \subseteq S_5\]
	of End$(\mathbb{Z}_5)$
	where $S_0=\emptyset$ and $S_i=\{\sigma_1,...,\sigma_i\}$ for $i=1,...,5$, and the figure-8 knot $L$. The set of $\mathbb{Z}_5$-colorings of $L$ is
	\begin{align*}
		\text{Col}_{\mathbb{Z}_5}(L)=\{&f_0=[5,5,5,5,5,5,5,5],
		f_1=[1,4,1,4,1,4,1,4],
		f_2=[2,3,2,3,2,3,2,3],\\
		&f_3=[3,2,3,2,3,2,3,2],
		f_4=[4,1,4,1,4,1,4,1]\}
	\end{align*}
	as also depicted in Figure \ref{fig8}.
	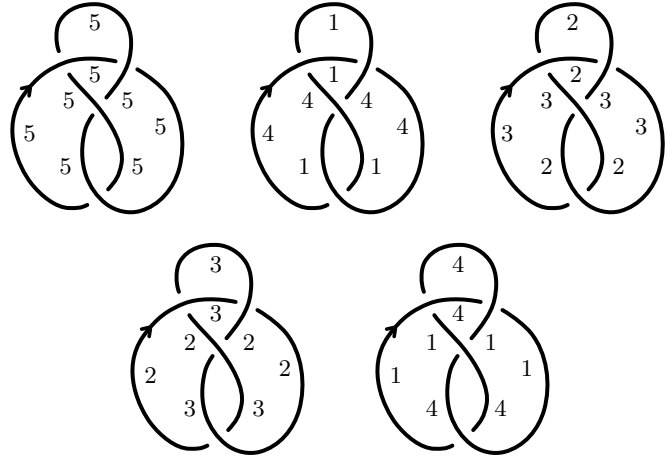
\begin{figure}
		\centering
		\begin{overpic}[scale=0.6]{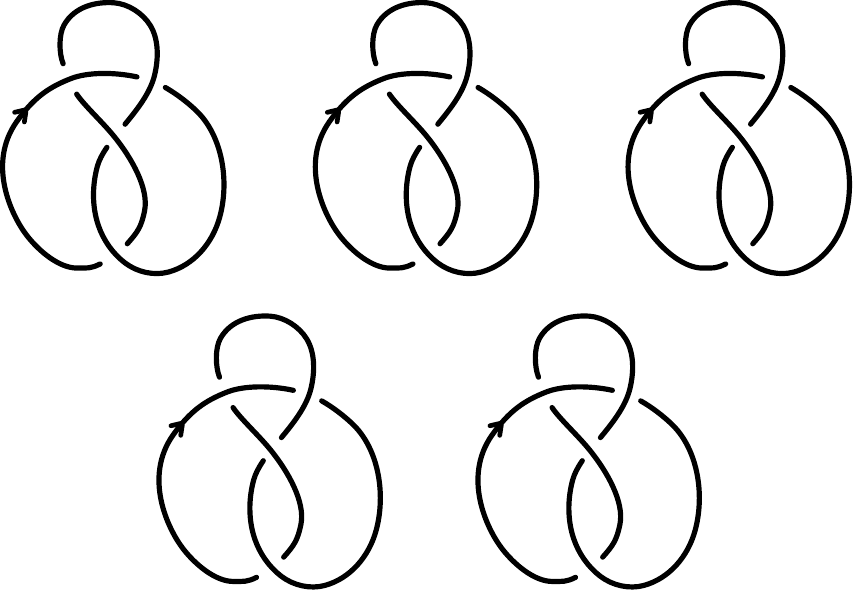}
			\put(2,48){\small $5$}
			\put(12,65){\small $5$}
			\put(8,53){\small $5$}
			\put(17,53){\small $5$}
			\put(12,57){\small $5$}
			\put(7.5,43){\small $5$}
			\put(18.5,43){\small $5$}
			\put(22,49){\small $5$}
			
			\put(38.5,48){\small $4$}
			\put(48.5,57){\small $1$}
			\put(59,49){\small $4$}
			\put(44,43){\small $1$}
			\put(53.5,53){\small $4$}
			\put(48.5,65){\small $1$}
			\put(44.5,53){\small $4$}
			\put(55,43){\small $1$}
			
			\put(75,48){\small $3$}
			\put(85.5,57){\small $2$}
			\put(95.5,49){\small $3$}
			\put(81,43){\small $2$}
			\put(90,53){\small $3$}
			\put(85,65){\small $2$}
			\put(81,53){\small $3$}
			\put(92,43){\small $2$}
			
			\put(20.5,11){\small $2$}
			\put(30.5,20.5){\small $3$}
			\put(41,12){\small $2$}
			\put(26.5,6){\small $3$}
			\put(35.5,16){\small $2$}
			\put(30.5,28){\small $3$}
			\put(26.5,16){\small $2$}
			\put(37,6){\small $3$}
			
			\put(58,11){\small $1$}
			\put(67.5,20.5){\small $4$}
			\put(78,12){\small $1$}
			\put(63.5,6){\small $4$}
			\put(72.5,16){\small $1$}
			\put(67.5,28){\small $4$}
			\put(63.5,16){\small $1$}
			\put(74,6){\small $4$}
		\end{overpic}
		\vspace{0.3 cm}
		\caption{$Z_5$-colorings of the figure-8 knot.}
		\label{fig8}
	\end{figure}
	\noindent The corresponding $\mathbb{Z}_5$-coloring quiver filtration of $L$ with respect to the filtration $S_\ast$ is,
	\[\mathcal{Q}_X^{S_\ast}(L):\quad \mathcal{Q}_X^{\emptyset}(L)=\mathcal{Q}_X^{S_0}(L)\subseteq \mathcal{Q}_X^{S_1}(L)\subseteq...\subseteq \mathcal{Q}_X^{S_5}(L)=\mathcal{Q}_X(L).\]
	We depict each $\mathbb{Z}_5$-coloring quiver $\mathcal{Q}_X^{S_i}(L)$ for $i=0,...,5$ in Figure \ref{filt}.
	
	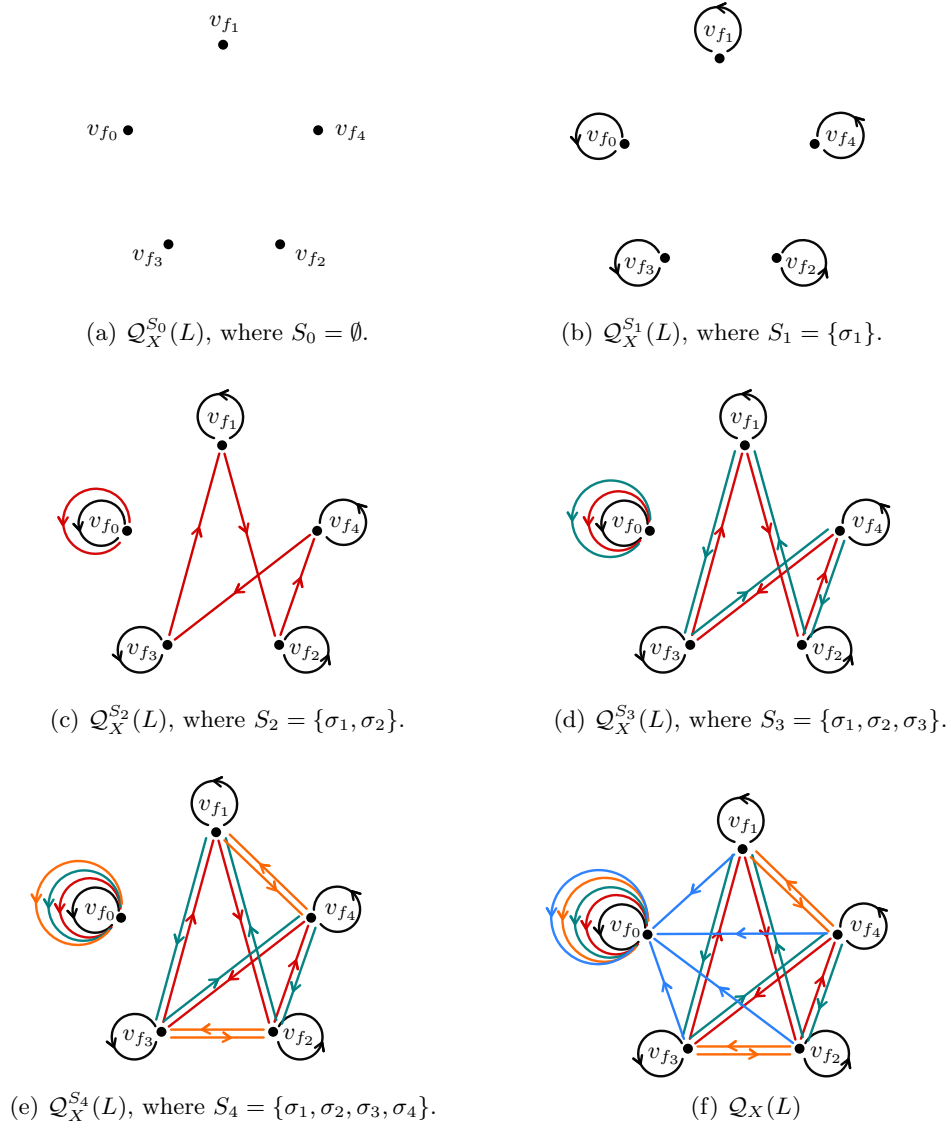
\begin{figure}
		\centering    
		\begin{subfigure}{0.4\textwidth}
			\centering
			\begin{overpic}[width=53mm]{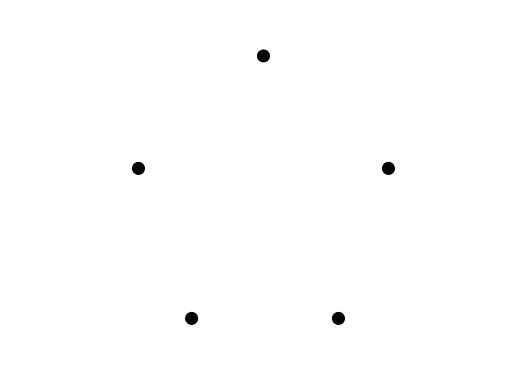}
				\put(46.5,65){\small$v_{f_1}$}
				\put(78,38){\small$v_{f_4}$}
				\put(68,6.6){\small$v_{f_2}$}
				\put(27.3,6.6){\small$v_{f_3}$}
				\put(16,38){\small$v_{f_0}$}
			\end{overpic}
			\vspace{0.2 cm}
			\caption{$\mathcal{Q}_X^{S_0}(L)$, where $S_0=\emptyset.$}
			\label{thtex}
		\end{subfigure}
		\begin{subfigure}
			{0.45\textwidth}
			\centering
			\begin{overpic}[width=53mm]{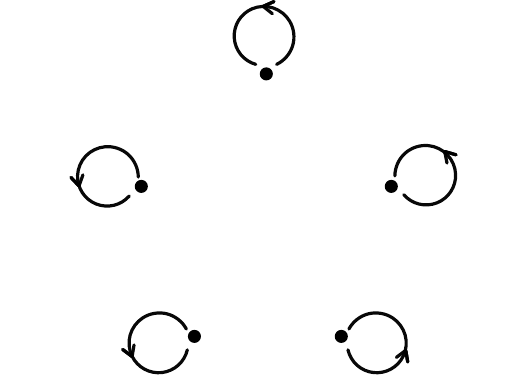}
				\put(46.5,62.5){\small$v_{f_1}$}
				\put(77,36.5){\small$v_{f_4}$}
				\put(67,4.7){\small$v_{f_2}$}
				\put(26.7,4.7){\small$v_{f_3}$}
				\put(17.3,36.5){\small$v_{f_0}$}
			\end{overpic}
			\vspace{0.2 cm}
			\caption{$\mathcal{Q}_X^{S_1}(L)$, where $S_1=\{\sigma_1\}$.}
			\label{balls}
		\end{subfigure}
		\begin{subfigure}
			{0.45\textwidth}
			\vspace{0.5 cm}
			\centering
			\begin{overpic}[width=53mm]{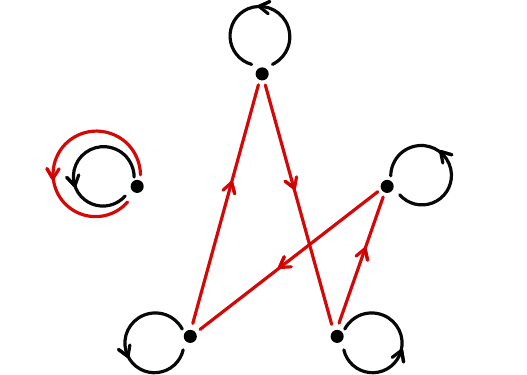}
				\put(46,62.5){\small$v_{f_1}$}
				\put(76.5,36.5){\small$v_{f_4}$}
				\put(66.5,5){\small$v_{f_2}$}
				\put(26,4.9){\small$v_{f_3}$}
				\put(16.5,36.5){\small$v_{f_0}$}
			\end{overpic}
			\vspace{0.2 cm}
			\caption{$\mathcal{Q}_X^{S_2}(L)$, where $S_2=\{\sigma_1,\sigma_2\}$.}
		\end{subfigure}  
		\begin{subfigure}
			{0.45\textwidth}
			\centering
			\begin{overpic}[width=53mm]{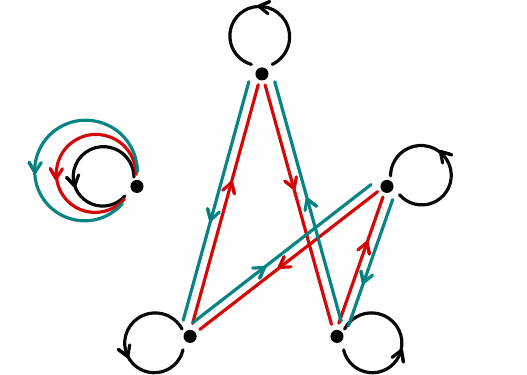}
				\put(46,62.5){\small$v_{f_1}$}
				\put(76.5,36.5){\small$v_{f_4}$}
				\put(66.5,5){\small$v_{f_2}$}
				\put(26,4.9){\small$v_{f_3}$}
				\put(16.5,36.5){\small$v_{f_0}$}
			\end{overpic}
			\vspace{0.2 cm}
			\caption{$\mathcal{Q}_X^{S_3}(L)$, where $S_3=\{\sigma_1,\sigma_2,\sigma_3\}$.}
		\end{subfigure}
		\begin{subfigure}
			{0.45\textwidth}
			\vspace{0.5 cm}
			\centering
			\begin{overpic}[width=53mm]{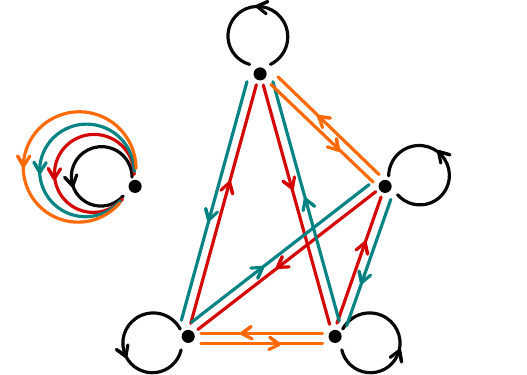}
				\put(45,62.5){\small$v_{f_1}$}
				\put(76.5,36.5){\small$v_{f_4}$}
				\put(66,5.3){\small$v_{f_2}$}
				\put(26,5.1){\small$v_{f_3}$}
				\put(16.5,36.5){\small$v_{f_0}$}
			\end{overpic}
			\vspace{0.2 cm}
			\caption{$\mathcal{Q}_X^{S_4}(L)$, where $S_4=\{\sigma_1,\sigma_2,\sigma_3,\sigma_4\}$.}
		\end{subfigure}
		\begin{subfigure}
			{0.45\textwidth}
			\centering
			\begin{overpic}[width=53mm]{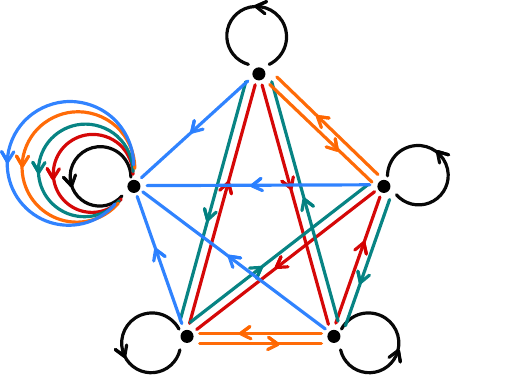}
				\put(45.5,62.5){\small$v_{f_1}$}
				\put(76,36.5){\small$v_{f_4}$}
				\put(66,5){\small$v_{f_2}$}
				\put(25.5,4.9){\small$v_{f_3}$}
				\put(16,36.5){\small$v_{f_0}$}
			\end{overpic}
			\caption{$\mathcal{Q}_X(L)$}
		\end{subfigure}
		\vspace{0.5 cm}
		\caption{A filtration of a biquandle coloring quiver.}
		\label{filt}
	\end{figure}
\end{example}

Example~\ref{filtex} illustrates how the combinatorial structure of the biquandle coloring quiver evolves under a filtration of the endomorphism set of the biquandle. This evolution will be captured algebraically using persistent homology in Section~\ref{persistsec}.

\begin{theorem}
	\label{filtinv}
	Let $X$ be a finite biquandle, and $S_\ast$ a filtration of End$(X)$. The $X$-coloring quiver filtration $\mathcal{Q}_X^{S_\ast}(L)$ is an invariant of links.
\end{theorem}

\begin{proof}
	Let $L_1$ and $L_2$ be virtually equivalent link diagrams. Then, there is a bijection \[\phi: \text{Col}_X(L_1)\rightarrow \text{Col}_X(L_2) \]
	between the sets of $X$-colorings. For each $i$, this bijection induces a quiver isomorphism
	\[\phi_i:\mathcal{Q}_X^{S_i}(L_1)\rightarrow \mathcal{Q}_X^{S_i}(L_2).\] 
	As the filtration $S_\ast$ is fixed, both quiver filtrations have the same number of stages, and for each $i$ the inclusions 
	\[\mathcal{Q}_X^{S_i}(L_j)\hookrightarrow \mathcal{Q}_X^{S_{i+1}}(L_j)\]
	are induced by the inclusions $S_i\hookrightarrow S_{i+1}$, for $j=1,2$. The maps $\phi_i$ are induced by the same bijection $\phi$ on $X$-colorings for all $i$. Hence, the restriction of $\phi_{i+1}$ to $\mathcal{Q}_X^{S_i}(L_1)$ coincides with $\phi_i$. It follows that the diagram 
	\begin{center}
		\begin{tikzcd} \mathcal{Q}_X^{S_i}(L_1) \arrow[hookrightarrow]{r} \arrow[d,"\phi_i"] & \mathcal{Q}_X^{S_{i+1}}(L_1) \arrow[d,"\phi_{i+1}"] \\ \mathcal{Q}_X^{S_i}(L_2) \arrow[hookrightarrow]{r} & \mathcal{Q}_X^{S_{i+1}}(L_2) \end{tikzcd}
	\end{center}
	commutes for all $i$. Therefore the filtrations  $\mathcal{Q}_X^{S_\ast}(L_1)$ and $\mathcal{Q}_X^{S_\ast}(L_2)$ are isomorphic. Hence $\mathcal{Q}_X^{S_\ast}(L)$ is an invariant of links.
\end{proof}  

\begin{remark}
	As a consequence of Theorem~\ref{filtinv}, any invariant of quivers, when applied to the quivers in the biquandle coloring quiver filtration $\mathcal{Q}_X^{S_\ast}(L)$, defines an ordered multiset invariant of links. As an immediate example, we apply the in-degree polynomial invariant to the biquandle coloring quiver filtrations below.  However, the true strength of biquandle coloring quiver filtration emerges when a functor is applied to it, as we demonstrate in Section~\ref{persistsec}.
\end{remark}

\begin{definition}
	Let $L$ be a link diagram, $X$ a finite biquandle, and $S_\ast$ a filtration of End$(X)$. The \textit{in-degree quiver polynomial multiset} of $L$ is defined as the ordered multiset
	\[\Phi_X^{\text{deg}^+,S_\ast}(L)=\left\{\Phi_X^{\text{deg}^+,S_i}(L)\right\}_{S_i\in S_\ast}.\]
\end{definition}

\begin{example}
	Consider the biquandle coloring quiver filtration $\mathcal{Q}_X^{S_\ast}(L)$ of Example~\ref{filtex} depicted in Figure~\ref{filt}. Then,
	\[\Phi_X^{\text{deg}^+,S_\ast}(L)=\{5,5u,5u^2,5u^3,5u^4,4u^4+u^9\}.\]
\end{example}

\begin{remark}
	Consider a filtration $S_\ast=\{S_i\}_{i=0,...,m}$
	of End$(X)$. It follows immediately that the $X$-coloring quiver filtration $\mathcal{Q}_X^{S_\ast}(L)$ is an enhancement of $\mathcal{Q}_X^{S_i}(L)$ for all $i=0,...,m$. Moreover, consider a subset $S\subseteq \text{End}(X)$. Then, for any filtration $S_\ast=\{S_i\}_{i=0,...,m}$
	of End$(X)$ such that $S=S_i$ for some $i$, $\mathcal{Q}_X^{S_\ast}(L)$ is an enhancement of $\mathcal{Q}_X^{S}(L)$.
\end{remark}

\subsection{Persistent $\boldsymbol{N}$-Directed Clique Homology of Biquandle Coloring Quivers}
\label{persistsec}
In this section, we turn our attention to special types of filtrations $S_\ast=\{S_i\}_{i=0,...,m}$ of End$(X)$ where $S_0$ is set to be the empty set. That is,
\[S_\ast: \quad \emptyset=S_0\subseteq...\subseteq S_m.\]
\begin{proposition}
	\label{subsub}
	Let $L$ be a link, $X$ a finite biquandle, $S_\ast$ a filtration of End$(X)$ and consider the $X$-coloring quiver filtration $\mathcal{Q}_X^{S_\ast}(L)$. For all $i\leq j$,
	\[\mathcal{K}^{(N)}(\mathcal{Q}_X^{S_i}(L))\subseteq\mathcal{K}^{(N)}(\mathcal{Q}_X^{S_j}(L)),\]
	that is, the associated $N$-directed clique complex of $\mathcal{Q}_X^{S_i}$ is a subcomplex of that of $\mathcal{Q}_X^{S_j}$. Hence, the associated $N$-directed clique complexes of the quivers in $\mathcal{Q}_X^{S_\ast}$ form the filtration 
	\begin{align}
		\mathcal{K}^{(N)}(\mathcal{Q}_X^{S_\ast}(L)):\quad \mathcal{K}^{(N)}(\mathcal{Q}_X^{\emptyset}(L))=\mathcal{K}^{(N)}(\mathcal{Q}_X^{S_0}(L))\subseteq...\subseteq \mathcal{K}^{(N)}(\mathcal{Q}_X^{S_m}(L)). \label{ndcfilt}
	\end{align}
\end{proposition}
\begin{proof}
	For $i\leq j$, we have $S_i\subseteq S_j$ and hence $\mathcal{Q}_X^{S_i}(L)$ is a subquiver of $\mathcal{Q}_X^{S_j}(L)$ with the same vertex set. Thus, the number of directed edges between vertices does not decrease as we move along the filtration. That is, if we denote the number of directed edges from $v_f$ to $v_g$ in $\mathcal{Q}_X^{S_k}$ by $e_k(v_f,v_g)$ for all $k$, then $e_i(v_f,v_g)\leq e_j(v_f,v_g)$. Therefore, $e_j(v_f,v_g)\geq N$ whenever $e_i(v_f,v_g)\geq N$. It follows that any abstract ordered simplex $(v_{f_0},...v_{f_n})$ of stage $i$ is also an abstract ordered simplex of stage $j$ of the corresponding abstract ordered simplicial complex sequence.
\end{proof}
\begin{remark}
	As stated in Remark~\ref{justremark}, one could have defined the $N$-directed clique complex by requiring exactly $N$ directed edges between vertices. However, Proposition~\ref{subsub} would not be true in that case. Indeed, as the filtration progresses, the number of edges between two vertices may increase from $N$ to $N+q$, for some $q>0$, causing simplices that previously satisfied the condition to no longer do so. Consequently, simplices may disappear in the associated $N$-clique complexes along the filtration of the biquandle coloring quiver. Thus, the resulting sequence of $N$-directed clique complexes would fail to define a filtration.
\end{remark}

We proceed by applying the homology functor $H_n$ to Filtration~\eqref{ndcfilt} 
and obtain the sequence 
\begin{center}
	\scalebox{0.94}{$\begin{aligned}
			H_n(\mathcal{K}^{(N)}(\mathcal{Q}_X^{S_\ast}(L)): H_n(\mathcal{K}^{(N)}(\mathcal{Q}_X^{\emptyset}(L)))=H_n(\mathcal{K}^{(N)}(\mathcal{Q}_X^{S_0}(L)))\rightarrow ... \rightarrow H_n(\mathcal{K}^{(N)}(\mathcal{Q}_X^{S_m}(L))),
		\end{aligned}$}
\end{center}
where the arrows represent group homomorphisms induced by inclusion maps. This sequence, by Definitions~\ref{homdef1} and \ref{homdef}, corresponds to the sequence
\begin{align*}
	H_n^{(N)}(L;X,S_\ast):\quad H_n^{(N)}(L;X,\emptyset)=H_n^{(N)}(L;X,S_0) \rightarrow ... \rightarrow H_n^{(N)}(L;X,S_m).
\end{align*}
\begin{definition}
	Let $L$ be a link diagram, $X$ a finite biquandle, $S_\ast$ a filtration of End$(X)$ and $N\geq 1$. For $i \leq j$, the \textit{$n$-th persistent $N$-directed clique homology group} of $L$ with respect to $X$ and $S_\ast$ is defined by
	\begin{align*}
		H_{n}^{(N),i,j}(L;X,S_\ast):=\text{im}\left(H_{n}^{(N)}(L;X,S_i)\rightarrow H_{n}^{(N)}(L;X,S_j)\right).
	\end{align*}
\end{definition}
\begin{theorem}
	\label{homeq}
	Let $X$ be a finite biquandle and $S_\ast$ a filtration of End$(X)$. For $i \leq j$, the $n$-th persistent $N$-directed clique homology group $H_{n}^{(N),i,j}(L;X,S_\ast)$ is an invariant of links for $n\geq 0$ and $N\geq1$.
\end{theorem}
\begin{proof}
	Let $L_1$ and $L_2$ be virtually equivalent links. By Theorem~\ref{filtinv}, the $X$-coloring quiver filtrations $\mathcal{Q}_X^{S_\ast}(L_1)$ and $\mathcal{Q}_X^{S_\ast}(L_2)$ are isomorphic as filtrations in the sense that for each $i$, there is a quiver isomorphism
	\[\varphi_i:\mathcal{Q}_X^{S_i}(L_1)\rightarrow \mathcal{Q}_X^{S_i}(L_2)\]
	such that the diagram
	\begin{center}
		\begin{tikzcd} \mathcal{Q}_X^{S_i}(L_1) \arrow[hookrightarrow]{r} \arrow[d,"\varphi_i"] & \mathcal{Q}_X^{S_j}(L_1) \arrow[d,"\varphi_j"] \\ \mathcal{Q}_X^{S_i}(L_2) \arrow[hookrightarrow]{r} & \mathcal{Q}_X^{S_j}(L_2) \end{tikzcd}
	\end{center}
	commutes for each $i\leq j$. Applying $H_n^{(N)}$, we obtain a commutative diagram
	
	\begin{center}
		\begin{tikzcd} H_n^{(N)}(\mathcal{Q}_X^{S_i}(L_1)) \arrow[r] \arrow[d,"\varphi_{i_\ast}","\cong"'] & H_n^{(N)}(\mathcal{Q}_X^{S_j}(L_1)) \arrow[d,"\cong","\varphi_{j_\ast}"'] \\ H_n^{(N)}(\mathcal{Q}_X^{S_i}(L_2)) \arrow[r] & H_n^{(N)}(\mathcal{Q}_X^{S_j}(L_2)) \end{tikzcd}
	\end{center}
	for each $i\leq j$, where commutativity follows from the functoriality of $H_n^{(N)}$ (see Remark~\ref{func}). It follows that $\varphi_{j\ast}$ induces an isomorphism between the images of the maps
	\[H_n^{(N)}(\mathcal{Q}_X^{S_i}(L_1))\rightarrow H_n^{(N)}(\mathcal{Q}_X^{S_j}(L_1)) \quad \text{and}\quad H_n^{(N)}(\mathcal{Q}_X^{S_i}(L_2))\rightarrow H_n^{(N)}(\mathcal{Q}_X^{S_j}(L_2)),\]
	hence, by Definition~\ref{homdef}, between the images of the maps
	\[H_n^{(N)}(L_1;X,S_i))\rightarrow H_n^{(N)}(L_1;X,S_j)) \quad \text{and}\quad H_n^{(N)}(L_2;X,S_i))\rightarrow H_n^{(N)}(L_2;X,S_j)),\]
	respectively. Therefore, 
	\[H_{n}^{(N),i,j}(L_1;X,S_\ast)\cong H_{n}^{(N),i,j}(L_2;X,S_\ast).\]
	Thus, the $n$-th persistent $N$-directed clique homology group is an invariant of links.
\end{proof}

As in Section~\ref{persistenthom}, the \textit{persistent $N$-directed clique homology barcode} is defined as the multiset of persistence intervals determined by these persistent homology groups and we denote it by $\mathcal{P}^{(N)}(L;X,S_\ast)$.
\begin{corollary}
	\label{barcodeinv}
	Let $X$ be a finite biquandle and $S_\ast$ a filtration of End$(X)$. The persistent $N$-directed clique homology barcode $\mathcal{P}^{(N)}(L;X,S_\ast)$ is an invariant of links.
\end{corollary}
\begin{proof}
	By Theorem~\ref{homeq}, the persistent $N$-directed clique homology groups are invariants of links. Since the persistence barcode is determined by these groups, it follows that $\mathcal{P}^{(N)}(L;X,S_\ast)$ is also an invariant of links.
\end{proof}

\begin{remark}
	The particular reason for the choice $S_0=\emptyset$ in the filtration $S_\ast$ is the following. Observe that $\mathcal{Q}_X^{\emptyset}(L)$ consists of vertices and no edges. Hence, 
	\[H_0^{(N)}(\mathcal{Q}_X^{\emptyset}(L))\cong \mathbb{Z}^{\Phi_X^{\mathbb{Z}}(L)} \] and its generators are in a correspondence with $X$-colorings of $L$. It follows that the number of persistence intervals of the form $[0,i)$, where $i=1,...,m$ or $i=\infty$, is equal to the number of vertices in the biquandle coloring quiver, hence to the number of $X$-colorings of $L$. Thus, $\mathcal{P}^{(N)}(L;X,S_*)$ enhances the biquandle counting invariant $\Phi_X^{\mathbb{Z}}(L)$ by encoding not only the number of colorings, but also how these colorings gradually interact under the filtration $S_\ast$ of biquandle endomorphisms.
\end{remark}

\begin{figure}
	\centering
	\includegraphics[width=0.55\linewidth]{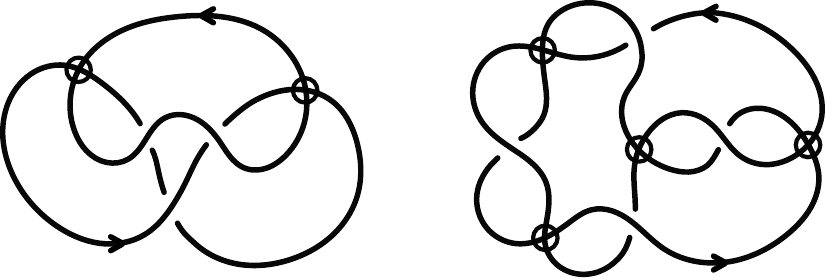}
	\vspace{0.3 cm}
	\caption{Two links $L_1$ (left-hand side) and $L_2$ (right-hand side) related to Example~\ref{ex2}}
	\label{vlink2}
\end{figure}

\begin{example}
	\label{ex2}
	Consider the link diagrams $L_1$ and $L_2$ in Figure~\ref{vlink2}, and the biquandle $\mathbb{Z}_6$ with the operation matrix
	\[\left[\begin{array}{c c c c c c | c c c c c c}
		1& 1& 1& 1& 1& 1& 1& 1& 1& 1& 1& 1\\
		2& 2& 2& 2& 2& 2& 2& 2& 4& 3& 4& 3\\
		3& 3& 3& 3& 3& 3& 3& 4& 3& 2& 2& 4\\
		4& 4& 4& 4& 4& 4& 4& 3& 2& 4& 3& 2\\
		5& 5& 5& 5& 5& 5& 5& 6& 6& 6& 5& 5\\
		6& 6& 6& 6& 6& 6& 6& 5& 5& 5& 6& 6\\
		
	\end{array}\right].\]
	Let $N=2$ and $S_\ast=\{S_i\}_{i=0,...,3}$ be a filtration of $\text{End}(\mathbb{Z}_6)$ where,
	\begin{align*}
		S_0&=\emptyset,\\
		S_1&=\{[6,6,6,6,6,6],[6,1,1,1,6,6]\},\\
		S_2&=S_1\cup\{[1,4,2,3,5,6],[6,1,1,1,1,1]\},\\
		S_3&=S_2\cup\{[1,6,6,6,1,1],[1,2,2,2,1,1],[1,4,4,4,1,1],[5,6,6,6,6,6]\}.
	\end{align*}
	Our Python computations show that the number of $\mathbb{Z}_6$-colorings of the links $L_1$ and $L_2$ are equal: $\Phi_{\mathbb{Z}_6}^{\mathbb{Z}}(L_1)= 30=\Phi_{\mathbb{Z}_6}^{\mathbb{Z}}(L_2)$. Persistence barcodes $\mathcal{P}^{(2)}(L_1;\mathbb{Z}_6,S_*)$ and $\mathcal{P}^{(2)}(L_2;\mathbb{Z}_6,S_*)$ are shown in Figure~\ref{pdiagrams}, where the numbers at the end of the blue bars represent the multiplicities of the corresponding persistence intervals. Observe that the persistence barcodes of $L_1$ and $L_2$ are not identical. In particular, the multiplicities of the intervals $[0,2),\;[0,3)$ and $[3,\infty)$ distinguish $L_1$ and $L_2$. Observe also that the total number of persistence intervals in $H_0^{(2)}$ is equal to the biquandle counting invariant in both barcodes, as expected. All homology computations in this and the following examples are carried out over $\mathbb{Z}_2$ for computational simplicity.
	
	\begin{figure}
		\centering    
		\begin{subfigure}{0.5\textwidth}
			\centering
			\begin{overpic}[width=0.85\linewidth]{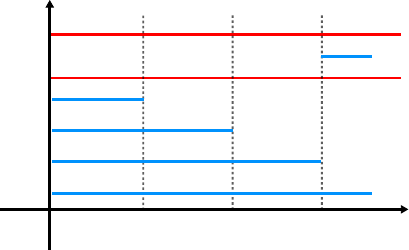}
				\put(0,25){\small $H_0^{(2)}$}
				\put(0,46){\small $H_2^{(2)}$}
				\put(59,28){\small$15$}
				\put(93,13){\small$1$}
				\put(93,46){\small$13$}
				\put(37,36){\small$8$}
				\put(81,20.5){\small$6$}
				\put(7,4){\small $0$}
				\put(34,4){\small $1$}
				\put(56,4){\small $2$}
				\put(78,4){\small $3$}
			\end{overpic}
			\caption{$\mathcal{P}^{(2)}(L_1;\mathbb{Z}_6,S_*)$}
			\vspace{0.3 cm}
		\end{subfigure}\hfill
		\begin{subfigure}{0.5\textwidth}
			\centering
			\begin{overpic}[width=0.85\linewidth]{barcode243.pdf}
				\put(0,25){\small $H_0^{(2)}$}
				\put(0,46){\small $H_2^{(2)}$}
				\put(59,28){\small$9$}
				\put(93,13){\small$1$}
				\put(93,46){\small$7$}
				\put(37,36){\small$8$}
				\put(81,20.5){\small$12$}
				\put(7,4){\small $0$}
				\put(34,4){\small $1$}
				\put(56,4){\small $2$}
				\put(78,4){\small $3$}
			\end{overpic}
			\caption{$\mathcal{P}^{(2)}(L_2;\mathbb{Z}_6,S_*)$}
			\vspace{0.3 cm}
		\end{subfigure}
		\caption{Persistence barcodes of the links of Example~\ref{ex2}.}
		\label{pdiagrams}
	\end{figure}
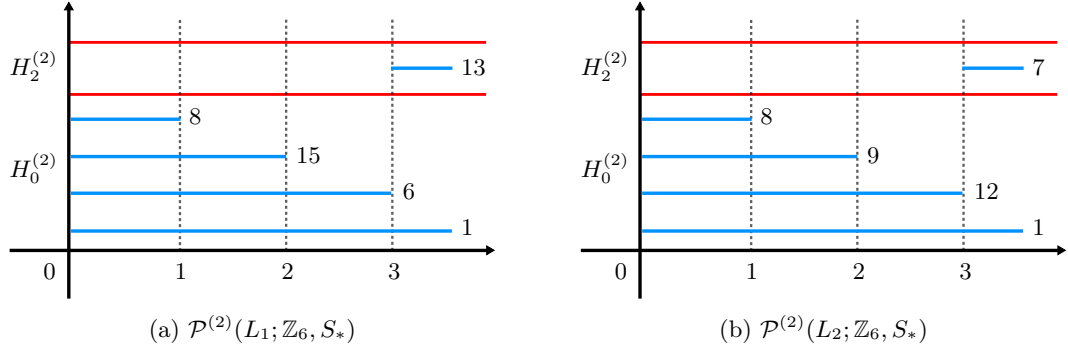
\end{example}

Example~\ref{ex2} shows that there exist links which are not distinguished by the biquandle counting invariant but the persistence barcode. Hence, the persistence barcode is a proper enhancement of the biquandle counting invariant.

\begin{remark} 
	In contrast with the general setting described in Section~\ref{persistenthom}, the filtration considered in our setting is defined from the biquandle coloring quiver associated to a link and determined entirely by the filtration of the biquandle endomorphism set End$(X)$. Hence, it is preserved under Reidemeister moves. Consequently, the resulting persistence groups and barcodes are invariants of links. That is, any difference between two barcodes, even in a single (possibly short) persistence interval, implies that the corresponding links are not equivalent. Hence, in our setting, not only the long-persisting homology classes but all homology classes provide meaningful information.
\end{remark} 

As demonstrated in this section, persistence barcodes provide an invariant of links whose strength and computability can be adjusted by suitable choices of the filtration $S_\ast$ and $N$. However, we can further enhance persistence barcodes by considering the following observation. Consider the filtrations of two abstract ordered simplicial complexes $\mathcal{K}$ and $\mathcal{K'}$ 
\begin{center}
	\scalebox{1}{$\begin{aligned}
			&\mathcal{K}_0=\{\{(v_0),(v_1),(v_2)\}\}\subseteq \mathcal{K}_1=\{\{(v_0),(v_1),(v_2)\},\{(v_0,v_1),(v_0,v_2),(v_1,v_2)\}\},\\
			&\mathcal{K}_0'=\{\{(v_0),(v_1),(v_2)\}\}\subseteq \mathcal{K}_1'=\{\{(v_0),(v_1),(v_2)\},\{(v_0,v_1),(v_0,v_2),(v_1,v_2)\},\{(v_0,v_1,v_2)\}\},
		\end{aligned}$}
\end{center}
respectively.
\begin{figure}
	\centering
	\begin{overpic}[width=0.4\linewidth]{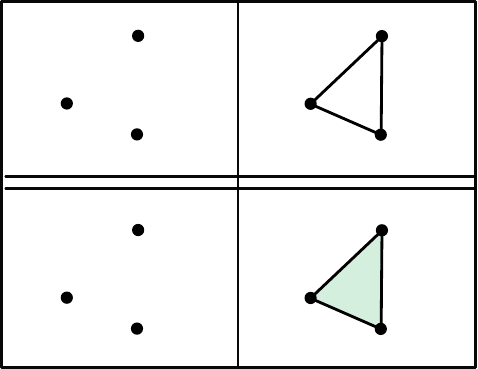}
		\put(3,45){\small$|\mathcal{K}_0|$}
		\put(53,45){\small$|\mathcal{K}_1|$}
		\put(3,5){\small$|\mathcal{K}_0'|$}
		\put(53,5){\small$|\mathcal{K}_1'|$}
	\end{overpic}
	\vspace{0.3 cm}
	\caption{Filtrations of two abstract ordered simplicial complexes $\mathcal{K}$ and $\mathcal{K}'$.}
	\label{bornd}
\end{figure}
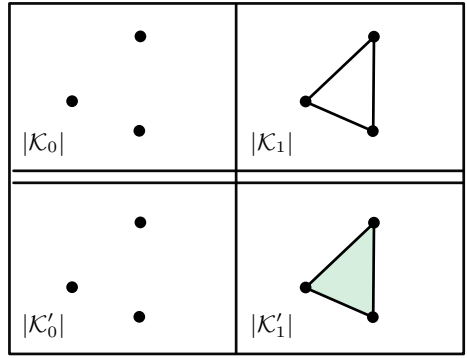
The geometric realizations of these filtrations are depicted in Figure~\ref{bornd}. The corresponding chain complexes for $\mathcal{K}_0$ and $\mathcal{K}_1$ are
\begin{align*}
	&C_\ast(\mathcal{K}_0):\quad ...\rightarrow \{0\}\rightarrow \qquad\qquad\;\;\{0\}\qquad\qquad\;\;\; \rightarrow \langle(v_0),(v_1),(v_2)\rangle \rightarrow \{0\},\\
	&C_\ast(\mathcal{K}_1):\quad ...\rightarrow\{0\}\rightarrow \langle (v_0,v_1),(v_0,v_2),(v_1,v_2)\rangle \rightarrow \langle(v_0),(v_1),(v_2)\rangle \rightarrow \{0\}.
\end{align*}
In $C_\ast(\mathcal{K}_1)$, it is clear that $(v_0,v_1)-(v_0,v_2)+(v_1,v_2)\in$ ker$(\partial_1)$. Thus, as we move from $\mathcal{K}_0$ to $\mathcal{K}_1$, an element $(v_0,v_1)-(v_0,v_2)+(v_1,v_2)$ of ker$(\partial_1)$ is born and since im$(\partial_2)=0$, $(v_0,v_1)-(v_0,v_2)+(v_1,v_2)$ corresponds to a nontrivial homology class. That is, a nontrivial homology class is born at the second stage of the filtration.

On the other hand, the corresponding chain complexes for $\mathcal{K}_0'$ and $\mathcal{K}_1'$ are
\begin{align*}
	&C_\ast(\mathcal{K}_0'): ...\rightarrow \quad\;\;\;\{0\}\qquad\rightarrow \qquad\qquad\;\;\{0\}\qquad\qquad\;\;\; \rightarrow \langle(v_0),(v_1),(v_2)\rangle \rightarrow \{0\},\\
	&C_\ast(\mathcal{K}_1'): ...\rightarrow \langle (v_0,v_1,v_2)\rangle\rightarrow \langle (v_0,v_1),(v_0,v_2),(v_1,v_2)\rangle \rightarrow \langle(v_0),(v_1),(v_2)\rangle \rightarrow \{0\}.
\end{align*}
In $C_\ast(\mathcal{K}_1')$, we have $(v_0,v_1)-(v_0,v_2)+(v_1,v_2)\in \operatorname{ker}(\partial_1')$. However, this time, $(v_0,v_1)-(v_0,v_2)+(v_1,v_2)\in\operatorname{im}(\partial_2')$ since it is also the boundary of the 2-simplex $(v_0,v_1,v_2)$ which is born in $\mathcal{K}_1$ as well. Thus, as we move from $\mathcal{K}_0'$ to $\mathcal{K}_1'$, a homology class $(v_0,v_1)-(v_0,v_2)+(v_1,v_2)+\text{im}(\partial_2')$ is born, but only as a trivial homology class. We call such homology classes \textit{stillborn homology classes}; that is, homology classes that are born at a stage of the filtration but are already trivial at that stage.

In the standard persistence barcode, the stillborn homology classes are not recorded and are therefore invisible in the persistence barcode. In our setting, however, they carry meaningful combinatorial information, and we encode them as matrices.

\begin{definition}
	Let $L$ be a link, $X$ a finite biquandle, $S_\ast$ a filtration of End$(X)$ and $N\geq 1$. The \textit{stillborn matrix} associated to $L$ with respect to $X$ and $S_\ast$ is defined as \[\mathcal{M}^{(N)}(L;X,S_\ast)=(a_{ij})_{i,j\geq 0},\] where $a_{ij}$ is the number of stillborn homology classes of dimension $i$ and stage $j$ of the filtration $\mathcal{Q}_{X}^{S_\ast}(L)$.
\end{definition}

Although we define the stillborn matrix for links through the biquandle coloring quiver filtrations, the construction applies to filtrations of simplicial complexes. We therefore illustrate the concept with the following simple example.

\begin{example}
	Consider the filtrations of simplicial complexes $\mathcal{K}$ and $\mathcal{K}'$ depicted in Figure~\ref{bornd}. The stillborn matrices of $\mathcal{K}$ and $\mathcal{K}'$, indexed by dimensions $0,1$ and filtration stages $0,1$, are
	\[\mathcal{M}^{(1)}(\mathcal{K})=\left[\begin{array}{c c}
		0& 0\\
		0& 0
	\end{array}\right], \quad \mathcal{M}^{(1)}(\mathcal{K}')=\left[\begin{array}{c c}
		0& 0\\
		0& 1
	\end{array}\right].\]
\end{example}
\begin{theorem}
	\label{deadinv}
	Let $X$ be a finite biquandle, $S_\ast$ a filtration of End$(X)$ and $N\geq 1$. Then, the stillborn matrix $\mathcal{M}^{(N)}(L;X,S_\ast)$ is an invariant of links.
\end{theorem}

\begin{proof}
	Let $L_1$ and $L_2$ be two virtually equivalent links. By Theorem ~\ref{filtinv}, there is an isomorphism 
	$\varphi^i : \mathcal{Q}_X^{S_i}(L_1) \to \mathcal{Q}_X^{S_i}(L_2)$ for all $i$, 
	compatible with the inclusion maps. Since each $\varphi^i$ is a quiver isomorphism, it induces isomorphisms 
	$\varphi^i_j : C_j(\mathcal{Q}_X^{S_i}(L_1)) \to C_j(\mathcal{Q}_X^{S_i}(L_2))$ 
	between the chain groups, and these maps commute with the boundary maps, 
	as illustrated by the following diagram.
	\begin{center}
		\begin{tikzcd} C_{j+1}(\mathcal{Q}_X^{S_i}(L_1)) \arrow[r,"\partial_{j+1}"] \arrow[d,"\varphi^{i}_{j+1}"'] & C_j(\mathcal{Q}_X^{S_i}(L_1)) \arrow[r,"\partial_{j}"]\arrow[d,"\varphi^{i}_{j}"'] & C_{j-1}(\mathcal{Q}_X^{S_i}(L_1))\arrow[d,"\varphi^{i}_{j-1}"']\\ C_{j+1}(\mathcal{Q}_X^{S_i}(L_2)) \arrow[r,"\partial_{j+1}"] & C_{j}(\mathcal{Q}_X^{S_i}(L_2))\arrow[r,"\partial_{j}"]& C_{j-1}(\mathcal{Q}_X^{S_i}(L_2)) \end{tikzcd}
	\end{center}
	It follows that $\varphi^i_j$ induces isomorphisms between the kernels and images 
	of the boundary maps. In particular, the subspace 
	$\ker(\partial_j) \,\cap\, \operatorname{im}(\partial_{j+1})$ corresponding to stillborn 
	homology classes is preserved for all $j$. Hence, the number of stillborn homology classes at each dimension and at each stage 
	of the filtration is invariant under virtual equivalence.
\end{proof}

We now combine the information obtained from the persistence barcode and the stillborn matrix to define a stronger invariant of links.

\begin{definition}
	Let $L$ be a link, $X$ a finite biquandle, $S_\ast$ a filtration of End$(X)$ and $N\geq 1$. We define the \textit{persistence pair} of $L$ as the pair $(\mathcal{P}^{(N)}(L;X,S_\ast),\mathcal{M}^{(N)}(L;X,S_\ast))$ and denote it by $\mathcal{PM}^{(N)}(L;X,S_\ast)$.
\end{definition}

\begin{corollary}
	Let $X$ be a finite biquandle, $S_\ast$ a filtration of End$(X)$ and $N\geq 1$. Then, the persistence pair $\mathcal{PM}^{(N)}(L;X,S_\ast)$ is an invariant of links.
\end{corollary}
\begin{proof}
	The proof follows immediately from Corollary \ref{barcodeinv} (invariance of persistence barcode) and Theorem \ref{deadinv} (invariance of stillborn matrix).
\end{proof}

\begin{remark}
	The construction of the persistence pair invariant $\mathcal{PM}^{(N)}(L;X,S_\ast)$ does not depend essentially on the use of $N$-directed clique homology. The role of $N$-directed clique homology in the above construction can be replaced by any homology theory $\mathcal{H}_\ast$ of quivers that is functorial with respect to inclusions of quivers. Then, applying the same construction in this section, one obtains the corresponding persistent $\mathcal{H}_\ast$-homology groups, persistence barcodes, and stillborn matrices, and the invariance results extend to this setting.
\end{remark}

\begin{figure}
	\centering
	\includegraphics[width=0.6\linewidth]{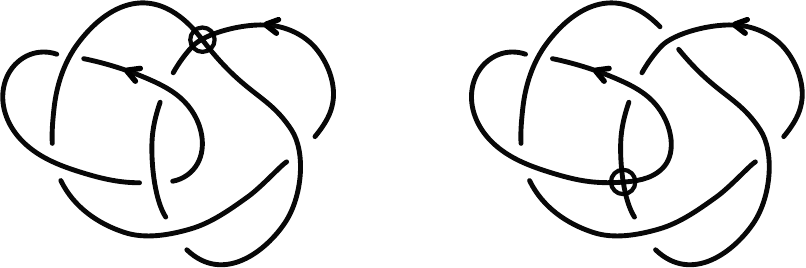}
	\vspace{0.3 cm}
	\caption{Two links $L_1$ (left-hand side) and $L_2$ (right-hand side) related to Example~\ref{ex3}}
	\label{vlink3}
\end{figure}

\begin{example}
	\label{ex3}
	Consider the link diagrams $L_1$ and $L_2$ in Figure~\ref{vlink3}, and the biquandle $\mathbb{Z}_6$ with the operation matrix
	\[\left[\begin{array}{c c c c c c | c c c c c c}
		1& 3& 2& 3& 1& 2& 1& 1& 1& 1& 1& 1\\
		3& 2& 1& 1& 2& 3& 2& 2& 2& 2& 2& 2\\
		2& 1& 3& 2& 3& 1& 3& 3& 3& 3& 3& 3\\
		6& 6& 6& 4& 4& 4& 4& 4& 4& 4& 4& 4\\
		5& 5& 5& 5& 5& 5& 5& 5& 5& 5& 5& 5\\
		4& 4& 4& 6& 6& 6& 6& 6& 6& 6& 6& 6\\
		
	\end{array}\right].\]
	Let $N=1$ and $S_\ast=\{S_i\}_{i=0,...,3}$ be a filtration of $\text{End}(\mathbb{Z}_6)$ where,
	\begin{align*}
		S_0&=\emptyset,\\
		S_1&=\{[6,6,6,6,6,6]\},\\
		S_2&=\{[6,6,6,6,6,6],[6,6,6,6,4,6]\},\\
		S_3&=\{[6,6,6,6,6,6],[6,6,6,6,4,6],[3,1,2,4,5,6],[4,4,4,5,6,5]\}.
	\end{align*}
	Our Python computations show that the number of $\mathbb{Z}_6$-colorings of the links $L_1$ and $L_2$ are equal: $\Phi_{\mathbb{Z}_6}^{\mathbb{Z}}(L_1)= 24=\Phi_{\mathbb{Z}_6}^{\mathbb{Z}}(L_2)$. Visualizations of the persistence pairs $\mathcal{PM}^{(1)}(L_1;\mathbb{Z}_6,S_*)$ and $\mathcal{PM}^{(1)}(L_2;\mathbb{Z}_6,S_*)$ as the persistence barcodes and the corresponding stillborn matrices are shown in Figure~\ref{pdiagrams2}. Observe that the persistence barcodes are identical. However, the stillborn matrices differ in the $(1,2)$- and $(2,3)$-entries, distinguishing these two links.
	
	\begin{figure}
		\centering    
		\begin{subfigure}{0.5\textwidth}
			\centering
			\begin{overpic}[width=0.9\linewidth]{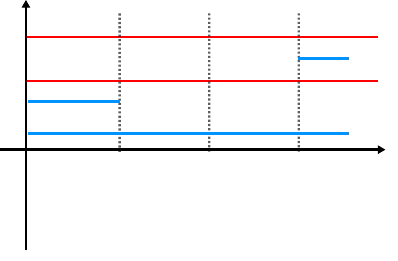}
				\put(-6,34){\small $H_0^{(1)}$}
				\put(-6,49){\small $H_2^{(1)}$}
				\put(90,49){\small$2$}
				\put(32,38){\small$23$}
				\put(90,30.5){\small$1$}
				\put(1.5,22){\small $0$}
				\put(29,22){\small $1$}
				\put(51.5,22){\small $2$}
				\put(74,22){\small $3$}
				\put(14,10){\footnotesize$\left[\begin{array}{c c c c}
						0& 0& 0& 0\\
						0& 0& 11& 38\\
						0& 0& 0& 21
					\end{array}\right]$}
			\end{overpic}
			\caption{$\mathcal{PM}^{(1)}(L_1;\mathbb{Z}_6,S_*)$}
			\vspace{0.3 cm}
		\end{subfigure}\hfill
		\begin{subfigure}{0.5\textwidth}
			\centering
			\begin{overpic}[width=0.9\linewidth]{persinv.pdf}
				\put(-6,34){\small $H_0^{(1)}$}
				\put(-6,49){\small $H_2^{(1)}$}
				\put(90,49){\small$2$}
				\put(32,38){\small$23$}
				\put(90,30.5){\small$1$}
				\put(1.5,22){\small $0$}
				\put(29,22){\small $1$}
				\put(51.5,22){\small $2$}
				\put(74,22){\small $3$}
				\put(14,10){\footnotesize$\left[\begin{array}{c c c c}
						0& 0& 0& 0\\
						0& 0& 17& 38\\
						0& 0& 0& 27
					\end{array}\right]$}
			\end{overpic}
			\caption{$\mathcal{PM}^{(1)}(L_2;\mathbb{Z}_6,S_*)$}
			\vspace{0.3 cm}
		\end{subfigure}
		\caption{Visualizations of the persistence pair invariant values of the links in Example~\ref{ex3}, showing the persistence barcodes and the corresponding stillborn matrices.}
		\label{pdiagrams2}
	\end{figure}
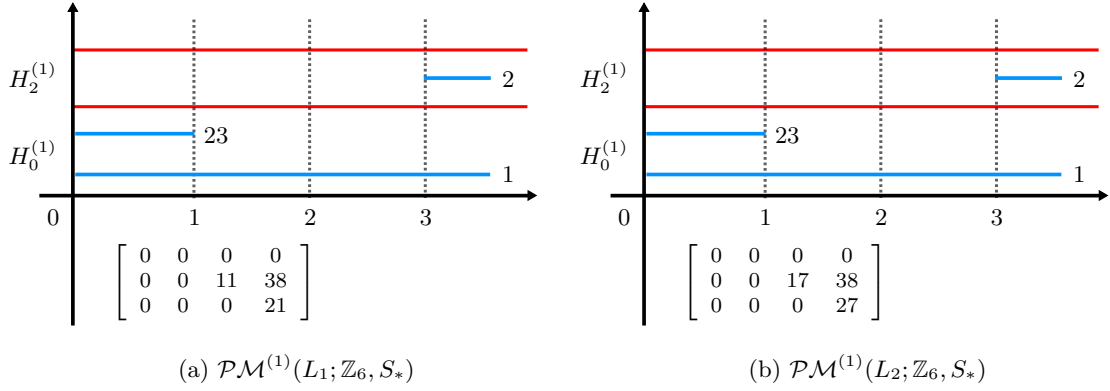
\end{example}

Example~\ref{ex3} shows that there exist links which are distinguished by neither the biquandle counting invariant nor the persistence barcode invariant, but the persistence pair invariant is able to distinguish these two links. Hence, the persistence pair invariant is a proper enhancement of the persistence barcode invariant of links.

\begin{example}
	Consider the biquandle $\mathbb{Z}_5$ with the operation matrix
	\[\left[\begin{array}{c c c c c | c c c c c}
		3& 4& 1& 4& 4& 3& 4& 3& 3& 4\\
		2& 5& 2& 2& 5& 5& 5& 5& 5& 5\\
		1& 1& 4& 3& 1& 4& 1& 4& 4& 1\\
		4& 3& 3& 1& 3& 1& 3& 1& 1& 3\\
		5& 2& 5& 5& 2& 2& 2& 2& 2& 2
	\end{array}\right].\]
	Let $N=1$ and $S_\ast=\{S_i\}_{i=0,...,4}$ a filtration of $\text{End}(\mathbb{Z}_6)$ where $S_0=\emptyset$, and 
	\[            S_1=\{[3,2,4,1,5]\},\;
	S_2=S_1\cup\{[4,2,1,3,5]\},\;
	S_3=S_2\cup\{[3,5,4,1,2]\},\;
	S_4=S_3\cup\{[4,5,1,3,2]\}.
	\]
	We use our Python code to compute persistent pair invariant values of some links and we present the results in Table~\ref{table3}. The naming conventions and the sources of these link diagrams are mentioned in Appendix~\ref{apx}.
	
	Observe that the links 2.3.1 and 2.3.9 in Table~\ref{table3} are not distinguished by the stillborn matrix invariant $\mathcal{M}^{(1)}(L;\mathbb{Z}_5,S_\ast)$ but the persistence pair invariant $\mathcal{PM}^{(1)}(L;\mathbb{Z}_5,S_\ast)$ is able to distinguish these two links. Hence, the persistence pair invariant is a proper enhancement of the stillborn matrix invariant, as well. Observe also that the persistence barcodes of these two links do not only differ by the multiplicities of the persistence intervals: the persistence interval $[2,\infty)$ occurs in $H_1^{(1)}$ of $\mathcal{P}^{(1)}(2.3.9;\mathbb{Z}_5,S_\ast)$ but it is absent in $H_1^{(1)}$ of $\mathcal{P}^{(1)}(2.3.1;\mathbb{Z}_5,S_\ast)$.
\end{example}
\begin{table}
	\centering
	\setlength{\tabcolsep}{3pt}
	\renewcommand{\arraystretch}{0.8}
	
	\begin{tabular}{|>{\centering\arraybackslash}m{0.16\textwidth}|
			>{\centering\arraybackslash}m{0.34\textwidth}||>{\centering\arraybackslash}m{0.16\textwidth}|>{\centering\arraybackslash}m{0.34\textwidth}|}
		\hline
		\vspace{0.1 cm}\text{\small $L$} \vspace{0.05 cm}& \vspace{0.1 cm}\text{\small $\mathcal{PM}^{(1)}(L;\mathbb{Z}_5,S_\ast)$}\vspace{0.05 cm} & \vspace{0.1 cm}\text{\small $L$} \vspace{0.05 cm}& \vspace{0.1 cm}\text{\small $\mathcal{PM}^{(1)}(L;\mathbb{Z}_5,S_\ast)$}\vspace{0.05 cm}
		\tabularnewline
		\hline
		
		\begin{overpic}[width=\linewidth]{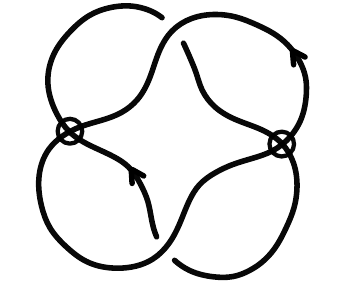}
			\put(35,-25){\footnotesize 2.2.1}
		\end{overpic}
		&
		\begin{overpic}[width=\linewidth]{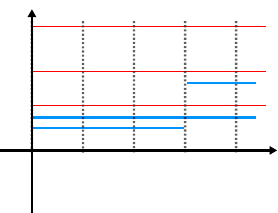}
			\put(1,30){\scalebox{0.65}{$H_0^{(1)}$} }
			\put(1,44){\scalebox{0.65}{$H_1^{(1)}$} }
			\put(1,59){\scalebox{0.65}{$H_2^{(1)}$} }
			
			\put(68.5,28.7){\scalebox{0.6}{$2$}}
			\put(93.5,33){\scalebox{0.6}{$2$}}
			\put(93.5,45.3){\scalebox{0.6}{$2$}}

			\put(7,17){\scalebox{0.6}{$0$}}
			\put(29,17){\scalebox{0.6}{$1$}}
			\put(48,17){\scalebox{0.6}{$2$}}
			\put(66,17){\scalebox{0.6}{$3$}}
			\put(84,17){\scalebox{0.6}{$4$}}
			\put(14,7){\scalebox{0.62}{
					$\left[\begin{array}{ccccc}
						0 & 0 & 0 & 0 & 0\\
						0 & 0 & 0 & 0 & 0\\
						0 & 0 & 0 & 0 & 0
					\end{array}\right]$
			}}
		\end{overpic}
		&
		\begin{overpic}[width=\linewidth]{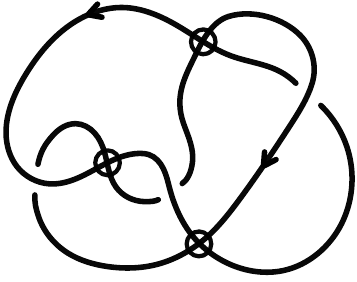}
			\put(35,-25){\footnotesize 2.3.1}
		\end{overpic}
		&
		\begin{overpic}[width=\linewidth]{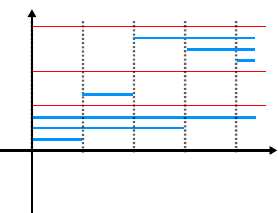}
			\put(1,30){\scalebox{0.65}{$H_0^{(1)}$} }
			\put(1,44){\scalebox{0.65}{$H_1^{(1)}$} }
			\put(1,59){\scalebox{0.65}{$H_2^{(1)}$} }
			
			\put(31.5,25){\scalebox{0.6}{$4$}}
			\put(68.5,28.7){\scalebox{0.6}{$1$}}
			\put(93.5,33){\scalebox{0.6}{$1$}}
			\put(50,41.5){\scalebox{0.6}{$2$}}
			\put(93.5,53.2){\scalebox{0.6}{$24$}}
			\put(93.5,57.5){\scalebox{0.6}{$1$}}
			\put(93.5,62.2){\scalebox{0.6}{$4$}}
			
			\put(7,17){\scalebox{0.6}{$0$}}
			\put(29,17){\scalebox{0.6}{$1$}}
			\put(48,17){\scalebox{0.6}{$2$}}
			\put(66,17){\scalebox{0.6}{$3$}}
			\put(84,17){\scalebox{0.6}{$4$}}
			\put(14,7){\scalebox{0.62}{
					$\left[\begin{array}{ccccc}
						0 & 0 & 0 & 0 & 0\\
						0 & 0 & 6 & 5 & 6\\
						0 & 0 & 0 & 0 & 0
					\end{array}\right]$
			}}
		\end{overpic}
		\tabularnewline
		\hline
		\begin{overpic}[width=\linewidth]{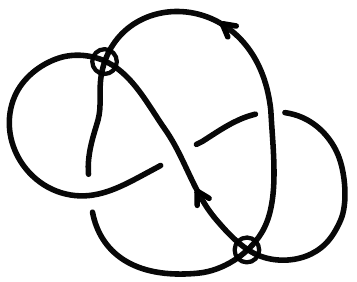}
			\put(35,-25){\footnotesize 2.3.2}
		\end{overpic}
		&
		\begin{overpic}[width=\linewidth]{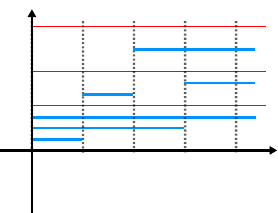}
			\put(1,30){\scalebox{0.65}{$H_0^{(1)}$} }
			\put(1,44){\scalebox{0.65}{$H_1^{(1)}$} }
			\put(1,59){\scalebox{0.65}{$H_2^{(1)}$} }
			
			\put(31.5,25){\scalebox{0.6}{$2$}}
			\put(68.5,28.7){\scalebox{0.6}{$2$}}
			\put(93.5,33){\scalebox{0.6}{$3$}}
			\put(50,41.5){\scalebox{0.6}{$1$}}
			\put(93.5,45){\scalebox{0.6}{$2$}}
			\put(93.5,57.5){\scalebox{0.6}{$2$}}
			
			\put(7,17){\scalebox{0.6}{$0$}}
			\put(29,17){\scalebox{0.6}{$1$}}
			\put(48,17){\scalebox{0.6}{$2$}}
			\put(66,17){\scalebox{0.6}{$3$}}
			\put(84,17){\scalebox{0.6}{$4$}}
			\put(14,7){\scalebox{0.62}{
					$\left[\begin{array}{ccccc}
						0 & 0 & 0 & 0 & 0\\
						0 & 0 & 3 & 0 & 0\\
						0 & 0 & 0 & 0 & 0
					\end{array}\right]$
			}}
		\end{overpic}
		&
		\begin{overpic}[width=\linewidth]{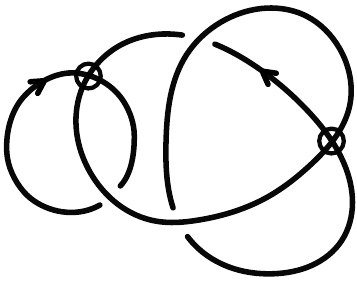}
			\put(35,-25){\footnotesize 2.3.6}
		\end{overpic}
		&
		\begin{overpic}[width=\linewidth]{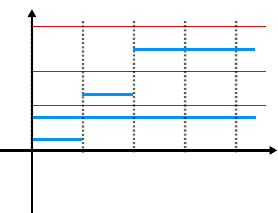}
			\put(1,30){\scalebox{0.65}{$H_0^{(1)}$} }
			\put(1,44){\scalebox{0.65}{$H_1^{(1)}$} }
			\put(1,59){\scalebox{0.65}{$H_2^{(1)}$} }
			
			\put(31.5,25){\scalebox{0.6}{$2$}}
			\put(93.5,33){\scalebox{0.6}{$1$}}
			\put(50,41.5){\scalebox{0.6}{$1$}}
			\put(93.5,57.5){\scalebox{0.6}{$2$}}
			
			\put(7,17){\scalebox{0.6}{$0$}}
			\put(29,17){\scalebox{0.6}{$1$}}
			\put(48,17){\scalebox{0.6}{$2$}}
			\put(66,17){\scalebox{0.6}{$3$}}
			\put(84,17){\scalebox{0.6}{$4$}}
			\put(14,7){\scalebox{0.62}{
					$\left[\begin{array}{ccccc}
						0 & 0 & 0 & 0 & 0\\
						0 & 0 & 3 & 0 & 0\\
						0 & 0 & 0 & 0 & 0
					\end{array}\right]$
			}}
		\end{overpic}
		\tabularnewline
		\hline
		\begin{overpic}[width=\linewidth]{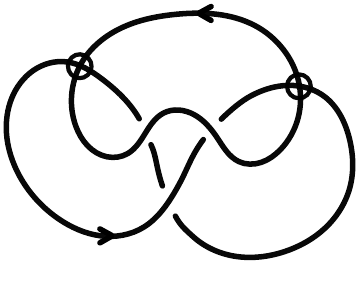}
			\put(35,-25){\footnotesize 2.3.9}
		\end{overpic}
		&
		\begin{overpic}[width=\linewidth]{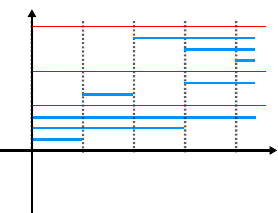}
			\put(1,30){\scalebox{0.65}{$H_0^{(1)}$} }
			\put(1,44){\scalebox{0.65}{$H_1^{(1)}$} }
			\put(1,59){\scalebox{0.65}{$H_2^{(1)}$} }
			
			\put(31.5,25){\scalebox{0.6}{$4$}}
			\put(68.5,28.7){\scalebox{0.6}{$3$}}
			\put(93.5,33){\scalebox{0.6}{$3$}}
			\put(50,41.5){\scalebox{0.6}{$2$}}
			\put(93.5,45){\scalebox{0.6}{$2$}}
			\put(93.5,53.2){\scalebox{0.6}{$24$}}
			\put(93.5,57.5){\scalebox{0.6}{$1$}}
			\put(93.5,62.2){\scalebox{0.6}{$4$}}
			
			\put(7,17){\scalebox{0.6}{$0$}}
			\put(29,17){\scalebox{0.6}{$1$}}
			\put(48,17){\scalebox{0.6}{$2$}}
			\put(66,17){\scalebox{0.6}{$3$}}
			\put(84,17){\scalebox{0.6}{$4$}}
			\put(14,7){\scalebox{0.62}{
					$\left[\begin{array}{ccccc}
						0 & 0 & 0 & 0 & 0\\
						0 & 0 & 6 & 5 & 6\\
						0 & 0 & 0 & 0 & 0
					\end{array}\right]$
			}}
		\end{overpic}
		&
		\begin{overpic}[width=\linewidth]{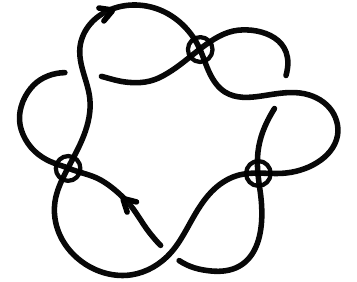}
			\put(35,-25){\footnotesize 2.3.10}
		\end{overpic}
		&
		\begin{overpic}[width=\linewidth]{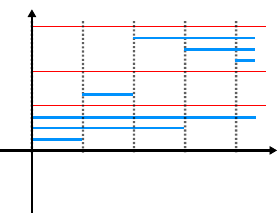}
			\put(1,30){\scalebox{0.65}{$H_0^{(1)}$} }
			\put(1,44){\scalebox{0.65}{$H_1^{(1)}$} }
			\put(1,59){\scalebox{0.65}{$H_2^{(1)}$} }
			
			\put(31.5,25){\scalebox{0.6}{$6$}}
			\put(68.5,28.7){\scalebox{0.6}{$1$}}
			\put(93.5,33){\scalebox{0.6}{$2$}}
			\put(50,41.5){\scalebox{0.6}{$3$}}
			\put(93.5,53.2){\scalebox{0.6}{$24$}}
			\put(93.5,57.5){\scalebox{0.6}{$1$}}
			\put(93.5,62.2){\scalebox{0.6}{$6$}}
			
			\put(7,17){\scalebox{0.6}{$0$}}
			\put(29,17){\scalebox{0.6}{$1$}}
			\put(48,17){\scalebox{0.6}{$2$}}
			\put(66,17){\scalebox{0.6}{$3$}}
			\put(84,17){\scalebox{0.6}{$4$}}
			\put(14,7){\scalebox{0.62}{
					$\left[\begin{array}{ccccc}
						0 & 0 & 0 & 0 & 0\\
						0 & 0 & 9 & 5 & 6\\
						0 & 0 & 0 & 0 & 0
					\end{array}\right]$
			}}
		\end{overpic}
		\tabularnewline
		\hline
		\begin{overpic}[width=\linewidth]{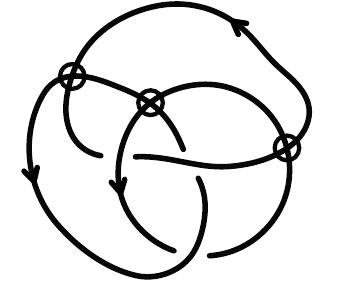}
			\put(35,-25){\footnotesize 3.3.0}
		\end{overpic}
		&
		\begin{overpic}[width=\linewidth]{b232-l6a40-810.pdf}
			\put(1,30){\scalebox{0.65}{$H_0^{(1)}$} }
			\put(1,44){\scalebox{0.65}{$H_1^{(1)}$} }
			\put(1,59){\scalebox{0.65}{$H_2^{(1)}$} }
			
			\put(31.5,25){\scalebox{0.6}{$2$}}
			\put(68.5,28.7){\scalebox{0.6}{$4$}}
			\put(93.5,33){\scalebox{0.6}{$5$}}
			\put(50,41.5){\scalebox{0.6}{$1$}}
			\put(93.5,45){\scalebox{0.6}{$4$}}
			\put(93.5,57.5){\scalebox{0.6}{$2$}}
			
			\put(7,17){\scalebox{0.6}{$0$}}
			\put(29,17){\scalebox{0.6}{$1$}}
			\put(48,17){\scalebox{0.6}{$2$}}
			\put(66,17){\scalebox{0.6}{$3$}}
			\put(84,17){\scalebox{0.6}{$4$}}
			\put(14,7){\scalebox{0.62}{
					$\left[\begin{array}{ccccc}
						0 & 0 & 0 & 0 & 0\\
						0 & 0 & 3 & 0 & 0\\
						0 & 0 & 0 & 0 & 0
					\end{array}\right]$
			}}
		\end{overpic}
		&
		\begin{overpic}[width=\linewidth]{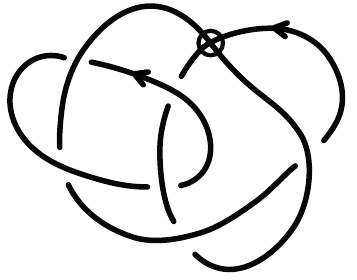}
			\put(35,-25){\footnotesize 2.6.0}
		\end{overpic}
		&
		\begin{overpic}[width=\linewidth]{b239-l7a31-l7a70.pdf}
			\put(1,30){\scalebox{0.65}{$H_0^{(1)}$} }
			\put(1,44){\scalebox{0.65}{$H_1^{(1)}$} }
			\put(1,59){\scalebox{0.65}{$H_2^{(1)}$} }
			
			\put(31.5,25){\scalebox{0.6}{$6$}}
			\put(68.5,28.7){\scalebox{0.6}{$3$}}
			\put(93.5,33){\scalebox{0.6}{$4$}}
			\put(50,41.5){\scalebox{0.6}{$3$}}
			\put(93.5,45){\scalebox{0.6}{$2$}}
			\put(93.5,53.2){\scalebox{0.6}{$24$}}
			\put(93.5,57.5){\scalebox{0.6}{$1$}}
			\put(93.5,62.2){\scalebox{0.6}{$6$}}
			
			\put(7,17){\scalebox{0.6}{$0$}}
			\put(29,17){\scalebox{0.6}{$1$}}
			\put(48,17){\scalebox{0.6}{$2$}}
			\put(66,17){\scalebox{0.6}{$3$}}
			\put(84,17){\scalebox{0.6}{$4$}}
			\put(14,7){\scalebox{0.62}{
					$\left[\begin{array}{ccccc}
						0 & 0 & 0 & 0 & 0\\
						0 & 0 & 9 & 5 & 6\\
						0 & 0 & 0 & 0 & 0
					\end{array}\right]$
			}}
		\end{overpic}
		\tabularnewline
		\hline
		\begin{overpic}[width=\linewidth]{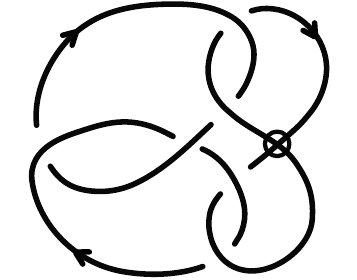}
			\put(35,-25){\footnotesize 3.6.0}
		\end{overpic}
		&
		\begin{overpic}[width=\linewidth]{b239-l7a31-l7a70.pdf}
			\put(1,30){\scalebox{0.65}{$H_0^{(1)}$} }
			\put(1,44){\scalebox{0.65}{$H_1^{(1)}$} }
			\put(1,59){\scalebox{0.65}{$H_2^{(1)}$} }
			
			\put(31.5,25){\scalebox{0.6}{$12$}}
			\put(68.5,28.7){\scalebox{0.6}{$7$}}
			\put(93.5,33){\scalebox{0.6}{$7$}}
			\put(50,41.5){\scalebox{0.6}{$6$}}
			\put(93.5,45){\scalebox{0.6}{$4$}}
			\put(93.5,53.2){\scalebox{0.6}{$72$}}
			\put(93.5,57.8){\scalebox{0.6}{$3$}}
			\put(93.5,62.2){\scalebox{0.6}{$12$}}
			
			\put(7,17){\scalebox{0.6}{$0$}}
			\put(29,17){\scalebox{0.6}{$1$}}
			\put(48,17){\scalebox{0.6}{$2$}}
			\put(66,17){\scalebox{0.6}{$3$}}
			\put(84,17){\scalebox{0.6}{$4$}}
			\put(14,7){\scalebox{0.62}{
					$\left[\begin{array}{ccccc}
						0 & 0 & 0 & 0 & 0\\
						0 & 0 & 18 & 15 & 18\\
						0 & 0 & 0 & 0 & 0
					\end{array}\right]$
			}}
		\end{overpic}
		&
		\begin{overpic}[width=\linewidth]{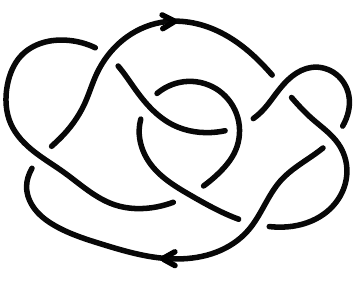}
			\put(35,-25){\footnotesize 1.8.0}
		\end{overpic}
		&
		\begin{overpic}[width=\linewidth]{b232-l6a40-810.pdf}
			\put(1,30){\scalebox{0.65}{$H_0^{(1)}$} }
			\put(1,44){\scalebox{0.65}{$H_1^{(1)}$} }
			\put(1,59){\scalebox{0.65}{$H_2^{(1)}$} }
			
			\put(31.5,25){\scalebox{0.6}{$6$}}
			\put(68.5,28.7){\scalebox{0.6}{$1$}}
			\put(93.5,33){\scalebox{0.6}{$4$}}
			\put(50,41.5){\scalebox{0.6}{$3$}}
			\put(93.5,45){\scalebox{0.6}{$1$}}
			\put(93.5,57.5){\scalebox{0.6}{$6$}}
			
			\put(7,17){\scalebox{0.6}{$0$}}
			\put(29,17){\scalebox{0.6}{$1$}}
			\put(48,17){\scalebox{0.6}{$2$}}
			\put(66,17){\scalebox{0.6}{$3$}}
			\put(84,17){\scalebox{0.6}{$4$}}
			\put(14,7){\scalebox{0.62}{
					$\left[\begin{array}{ccccc}
						0 & 0 & 0 & 0 & 0\\
						0 & 0 & 9 & 0 & 0\\
						0 & 0 & 0 & 0 & 0
					\end{array}\right]$
			}}
		\end{overpic}
		\tabularnewline
		\hline

	\end{tabular}
	\vspace{0.3 cm}
	\caption{Examples of link diagrams and their persistence pair invariant values.}
	\label{table3}
\end{table}
\section{Discussion}

The constructions in Section~\ref{mainsection1} suggest directions for further study. While in this paper we utilize N-directed clique homology, one may replace it with other suitable homology theories of quivers to obtain alternative persistence invariants of biquandle coloring quivers. In this sense, the choice of homology theory in the persistence construction in this paper can be regarded as an additional parameter. A comparison of the resulting invariants may provide insight into their relative strength in distinguishing links.

In particular, it would be interesting to investigate how different choices of homology theory affect the resulting persistence barcodes and stillborn homology matrices. Moreover, different homology theories may also provide varying levels of computational efficiency, which could play an important role in practical applications.

Another natural direction for future work is to apply $N$-directed clique homology, filtrations of biquandle coloring quivers, and the resulting persistent approach to classical and virtual knotoids (see \cite{gugumcu2017new, turaev2012knotoids}). It is expected that this application yields new and strong invariants for classical and virtual knotoids as well.

\section{Appendix}
	\label{apx}
	The computations of the invariants in this paper (especially persistence pair invariant) are hard to perform by hand. Thus, we use the Python codes we developed for the computations in the paper. The code itself is lengthy to include in the body of the paper. Hence, the code is uploaded to the online repository GitHub. The implementation we used to obtain the results in the paper, including the dataset of labeled peer codes of links and biquandles, is publicly available at  \footnote{\url{https://github.com/usot-lab/biquandle-coloring-quivers}}.
	
	We use the virtual link table of Bartholomew\footnote{A.~Bartholomew, \emph{A table of virtual links} (2022), available at \url{https://www.layer8.co.uk/maths/virtual-links/index.htm}.} and the Thistlethwaite link table\footnote{\emph{The Thistlethwaite Link Table}, available at \url{https://katlas.org/wiki/The_Thistlethwaite_Link_Table}} in Example~\ref{ex2} and Table~\ref{table3}. In particular, both of the links in Example~\ref{ex2} and the first six links in Table~\ref{table3} are from Bartholomew's table. In Bartholomew's table, virtual links are studied using their \textit{labeled peer codes}\footnote{A.~Bartholomew, \emph{The Labeled Peer Code For Knot and Link Diagrams} (2022), available at \url{https://www.layer8.co.uk/maths/download/labelled-peer-code.pdf}}. We name the link diagrams the same as in Bartholomew's table. We obtain the links 3.3.0 and 3.6.0 by switching some of the classical crossings of links in Thistlethwaite link table into virtual crossings. We used the same naming convention of Bartholomew's for the resulting link diagrams.
	
	The Python code consists of the following main files. 
	\begin{itemize}
		\item \textbf{dataset.py}\\
		This file includes data of labeled peer codes of links and biquandles, some of which we use in our examples in the body of the paper.
		\item \textbf{functions.py}\\
		This file consists of computational routines that we use to implement biquandle counting invariants, biquandle coloring quivers and homology computations. There are short explanations in each routine explaining its purpose and usage.
		\item \textbf{main.py}\\
		This file acts as an interface where we run computations in the command-line interface by specifying the necessary choices. In particular, we specify a link label and a biquandle to compute biquandle counting invariants or adjacency matrices of biquandle coloring quivers. We specify the choice for $N$, a subset $S$ of endomorphism set or a filtration $S_\ast$ of biquandle endomorphisms to compute homological invariants.
	\end{itemize}
	
	There is also a README file in the GitHub repository, which includes instructions for running the code. In particular, the commands to reproduce the results in the examples of the paper are provided in the file. As an example, Figure \ref{calcc} demonstrates computation results for the link 2.3.1 in Table~\ref{table3}.
	
	\begin{figure}
		\centering
		\includegraphics[width=1\linewidth]{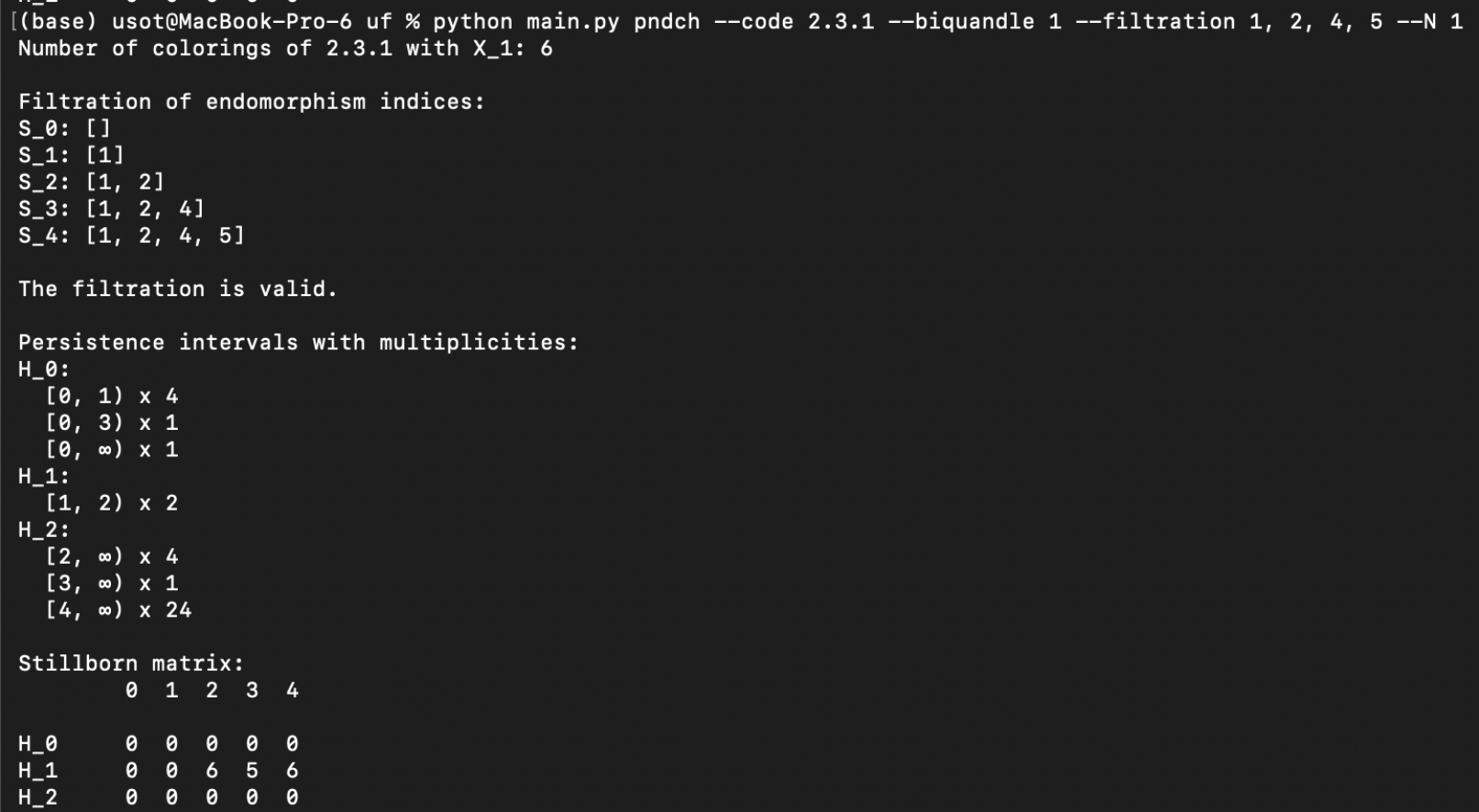}
		\vspace{0.2 cm}
		\caption{Persistence pair invariant calculation results for the link 2.3.1 with the Python code.}
		\label{calcc}
	\end{figure}
	
	\begin{remark}
		In Python, rows and columns of an $n\times n$ square matrix are indexed by the set $\{0, 1, ..., n-1\}$. However, in standard matrix convention, the index set is $\{1, ..., n\}$. Thus, operation matrices of biquandles differ by a shift in the paper and in the codes. A routine that converts a Python matrix to a standard matrix is provided in the \textbf{functions.py} file.
	\end{remark}

\clearpage

\bibliographystyle{plain}
\bibliography{references}

@article{aktas2019persistence,
	title={Persistence homology of networks: methods and applications},
	author={Aktas, Mehmet E and Akbas, Esra and Fatmaoui, Ahmed El},
	journal={Applied Network Science},
	volume={4},
	number={1},
	pages={1--28},
	year={2019},
	publisher={Springer}
}

@article{caputi2024hochschild,
	author = {Caputi, Luigi and Riihim{\"a}ki, Henri},
	title = {Hochschild homology, and a persistent approach via connectivity digraphs},
	fjournal = {Journal of Applied and Computational Topology},
	journal = {J. Appl. Comput. Topol.},
	issn = {2367-1726},
	volume = {8},
	number = {5},
	pages = {1121--1170},
	year = {2024},
	language = {English},
	doi = {10.1007/s41468-023-00118-9},
	zbMATH = {7955445},
	Zbl = {1551.05151}
}

@article{carlsson2009topology,
	author = {Carlsson, Gunnar},
	title = {Topology and data},
	fjournal = {Bulletin of the American Mathematical Society. New Series},
	journal = {Bull. Am. Math. Soc., New Ser.},
	issn = {0273-0979},
	volume = {46},
	number = {2},
	pages = {255--308},
	year = {2009},
	language = {English},
	doi = {10.1090/S0273-0979-09-01249-X},
	zbMATH = {5545159},
	Zbl = {1172.62002}
}

@article{carlsson2010zigzag,
	author = {Carlsson, Gunnar and de Silva, Vin},
	title = {Zigzag persistence},
	fjournal = {Foundations of Computational Mathematics},
	journal = {Found. Comput. Math.},
	issn = {1615-3375},
	volume = {10},
	number = {4},
	pages = {367--405},
	year = {2010},
	language = {English},
	doi = {10.1007/s10208-010-9066-0},
	zbMATH = {5796341},
	Zbl = {1204.68242}
}

@article{carstens2013persistent,
	author = {Carstens, C. J. and Horadam, K. J.},
	title = {Persistent homology of collaboration networks},
	fjournal = {Mathematical Problems in Engineering},
	journal = {Math. Probl. Eng.},
	issn = {1024-123X},
	volume = {2013},
	pages = {7},
	note = {Id/No 815035},
	year = {2013},
	language = {English},
	doi = {10.1155/2013/815035},
	zbMATH = {6372664},
	Zbl = {1299.91104}
}

@article{ceniceros2023psyquandle,
	author = {Ceniceros, Jose and Christiana, Anthony and Nelson, Sam},
	title = {Psyquandle coloring quivers},
	fjournal = {Journal of Knot Theory and its Ramifications},
	journal = {J. Knot Theory Ramifications},
	issn = {0218-2165},
	volume = {32},
	number = {11},
	pages = {18},
	note = {Id/No 2350073},
	year = {2023},
	language = {English},
	doi = {10.1142/S0218216523500736},
	zbMATH = {7782615},
	Zbl = {1530.57007}
}

@book{chazal2016structure,
	author    = {Chazal, Fr{\'e}d{\'e}ric and de Silva, Vin and Glisse, Marc and Oudot, Steve},
	title     = {The Structure and Stability of Persistence Modules},
	series    = {SpringerBriefs in Mathematics},
	year      = {2016},
	publisher = {Springer},
	address   = {Cham},
	doi       = {10.1007/978-3-319-42545-0}
}

@article{cho2019quandle,
	author = {Cho, Karina and Nelson, Sam},
	title = {Quandle coloring quivers},
	fjournal = {Journal of Knot Theory and its Ramifications},
	journal = {J. Knot Theory Ramifications},
	issn = {0218-2165},
	volume = {28},
	number = {1},
	pages = {12},
	note = {Id/No 1950001},
	year = {2019},
	language = {English},
	doi = {10.1142/S0218216519500019},
	zbMATH = {7023792},
	Zbl = {1420.57032}
}

@inproceedings{chowdhury2018persistent,
	title={Persistent path homology of directed networks},
	author={Chowdhury, Samir and M{\'e}moli, Facundo},
	booktitle={Proceedings of the twenty-ninth annual ACM-SIAM symposium on discrete algorithms},
	pages={1152--1169},
	year={2018},
	organization={SIAM}
}

@article{cohen2007stability,
	author = {Cohen-Steiner, David and Edelsbrunner, Herbert and Harer, John},
	title = {Stability of persistence diagrams},
	fjournal = {Discrete \& Computational Geometry},
	journal = {Discrete Comput. Geom.},
	issn = {0179-5376},
	volume = {37},
	number = {1},
	pages = {103--120},
	year = {2007},
	language = {English},
	doi = {10.1007/s00454-006-1276-5},
	zbMATH = {5126703},
	Zbl = {1117.54027}
}

@article{dey2022efficient,
	author = {Dey, Tamal K. and Li, Tianqi and Wang, Yusu},
	title = {An efficient algorithm for 1-dimensional (persistent) path homology},
	fjournal = {Discrete \& Computational Geometry},
	journal = {Discrete Comput. Geom.},
	issn = {0179-5376},
	volume = {68},
	number = {4},
	pages = {1102--1132},
	year = {2022},
	language = {English},
	doi = {10.1007/s00454-022-00430-8},
	zbMATH = {7628957},
	Zbl = {1547.68768}
}

@article{edelsbrunner2002topological,
	author = {Edelsbrunner, Herbert and Letscher, David and Zomorodian, Afra},
	title = {Topological persistence and simplification},
	fjournal = {Discrete \& Computational Geometry},
	journal = {Discrete Comput. Geom.},
	issn = {0179-5376},
	volume = {28},
	number = {4},
	pages = {511--533},
	year = {2002},
	language = {English},
	doi = {10.1007/s00454-002-2885-2},
	zbMATH = {1883340},
	Zbl = {1011.68152}
}

@article{edelsbrunner2008persistent,
	title={Persistent homology-a survey},
	author = {Edelsbrunner, Herbert and Harer, John},
	journal={Contemporary mathematics},
	volume={453},
	number={26},
	pages={257--282},
	year={2008},
	publisher={Providence, RI: American Mathematical Society}
}

@article{fenn2004biquandles,
	author = {Fenn, Roger and Jordan-Santana, Mercedes and Kauffman, Louis},
	title = {Biquandles and virtual links},
	fjournal = {Topology and its Applications},
	journal = {Topology Appl.},
	issn = {0166-8641},
	volume = {145},
	number = {1-3},
	pages = {157--175},
	year = {2004},
	language = {English},
	doi = {10.1016/j.topol.2004.06.008},
	zbMATH = {2136837},
	Zbl = {1063.57006}
}

@article{fenn1995trunks,
	author = {Fenn, Roger and Rourke, Colin and Sanderson, Brian},
	title = {Trunks and classifying spaces},
	fjournal = {Applied Categorical Structures},
	journal = {Appl. Categ. Struct.},
	issn = {0927-2852},
	volume = {3},
	number = {4},
	pages = {321--356},
	year = {1995},
	language = {English},
	doi = {10.1007/BF00872903},
	zbMATH = {854351},
	Zbl = {0853.55021}
}

@article{ghrist2008barcodes,
	author = {Ghrist, Robert},
	title = {Barcodes: the persistent topology of data},
	fjournal = {Bulletin of the American Mathematical Society. New Series},
	journal = {Bull. Am. Math. Soc., New Ser.},
	issn = {0273-0979},
	volume = {45},
	number = {1},
	pages = {61--75},
	year = {2008},
	language = {English},
	doi = {10.1090/S0273-0979-07-01191-3},
	zbMATH = {5493248},
	Zbl = {1391.55005}
}

@article{grigoryan2015cohomology,
	author = {Grigor'yan, Alexander and Lin, Yong and Muranov, Yuri and Yau, Shing-Tung},
	title = {Cohomology of digraphs and (undirected) graphs},
	fjournal = {The Asian Journal of Mathematics},
	journal = {Asian J. Math.},
	issn = {1093-6106},
	volume = {19},
	number = {5},
	pages = {887--931},
	year = {2015},
	language = {English},
	doi = {10.4310/AJM.2015.v19.n5.a5},
	zbMATH = {6534664},
	Zbl = {1329.05132}
}

@article{grigoryan2018path,
	author = {Grigor'yan, Alexander and Muranov, Yuri and Vershinin, Vladimir and Yau, Shing-Tung},
	title = {Path homology theory of multigraphs and quivers},
	fjournal = {Forum Mathematicum},
	journal = {Forum Math.},
	issn = {0933-7741},
	volume = {30},
	number = {5},
	pages = {1319--1337},
	year = {2018},
	language = {English},
	doi = {10.1515/forum-2018-0015},
	zbMATH = {6947027},
	Zbl = {1404.55010}
}

@article{gugumcu2017new,
	author = {G{\"u}g{\"u}mc{\"u}, Neslihan and Kauffman, Louis H.},
	title = {New invariants of knotoids},
	fjournal = {European Journal of Combinatorics},
	journal = {Eur. J. Comb.},
	issn = {0195-6698},
	volume = {65},
	pages = {186--229},
	year = {2017},
	language = {English},
	doi = {10.1016/j.ejc.2017.06.004},
	zbMATH = {6767276},
	Zbl = {1373.57012}
}

@book{hatcher2002algebraic,
	author    = {Hatcher, Allen},
	title     = {Algebraic Topology},
	year      = {2002},
	publisher = {Cambridge University Press},
	address   = {Cambridge},
	isbn      = {0-521-79540-0}
}

@article{hrencecin2007biquandles,
	author = {Hrencecin, David and Kauffman, Louis H.},
	title = {Biquandles for virtual knots},
	fjournal = {Journal of Knot Theory and its Ramifications},
	journal = {J. Knot Theory Ramifications},
	issn = {0218-2165},
	volume = {16},
	number = {10},
	pages = {1361--1382},
	year = {2007},
	language = {English},
	doi = {10.1142/S0218216507005841},
	zbMATH = {5274873},
	Zbl = {1161.57004}
}

@misc{lin2019weighted,
	author = {Yong Lin and Shiquan Ren and Chong Wang and Jie Wu},
	title = {Weighted {Path} homology of {Weighted} {Digraphs} and {Persistence}},
	year = {2019},
	howpublished = {Preprint, {arXiv}:1910.09891},
	url = {https://arxiv.org/abs/1910.09891},
	arXiv = {arXiv:1910.09891}
}

@article{lutgehetmann2020computing,
	title={Computing persistent homology of directed flag complexes},
	author={L{\"u}tgehetmann, Daniel and Govc, Dejan and Smith, Jason P and Levi, Ran},
	journal={Algorithms},
	volume={13},
	number={1},
	pages={19},
	year={2020},
	publisher={MDPI}
}

@article{masulli2016topology,
	title={The topology of the directed clique complex as a network invariant},
	author={Masulli, Paolo and Villa, Alessandro EP},
	journal={SpringerPlus},
	volume={5},
	number={1},
	pages={388},
	year={2016},
	publisher={Springer}
}

@article{kauffman2012introduction,
	author = {Kauffman, Louis H.},
	title = {Introduction to virtual knot theory},
	fjournal = {Journal of Knot Theory and its Ramifications},
	journal = {J. Knot Theory Ramifications},
	issn = {0218-2165},
	volume = {21},
	number = {13},
	pages = {1240007, 37},
	year = {2012},
	language = {English},
	doi = {10.1142/S021821651240007X},
	zbMATH = {6109758},
	Zbl = {1255.57005}
}

@article{kuperberg2003what,
	author = {Kuperberg, Greg},
	title = {What is a virtual link?},
	fjournal = {Algebraic \& Geometric Topology},
	journal = {Algebr. Geom. Topol.},
	issn = {1472-2747},
	volume = {3},
	pages = {587--591},
	year = {2003},
	language = {English},
	doi = {10.2140/agt.2003.3.587},
	url = {https://eudml.org/doc/123263},
	zbMATH = {1985391},
	Zbl = {1031.57010}
}

@article{ivanov2024simplicial,
	author = {Ivanov, Sergei O. and Pavutnitskiy, Fedor},
	title = {Simplicial approach to path homology of quivers, marked categories, groups and algebras},
	fjournal = {Journal of the London Mathematical Society. Second Series},
	journal = {J. Lond. Math. Soc., II. Ser.},
	issn = {0024-6107},
	volume = {109},
	number = {1},
	pages = {68},
	note = {Id/No e12812},
	year = {2024},
	language = {English},
	doi = {10.1112/jlms.12812},
	zbMATH = {7808985},
	Zbl = {1537.18014}
}

@book{elhamdadi2015quandles,
	author    = {Elhamdadi, Mohamed and Nelson, Sam},
	title     = {Quandles: An Introduction to the Algebra of Knots},
	year      = {2015},
	publisher = {American Mathematical Society},
	address   = {Providence, RI},
	series    = {Student Mathematical Library},
	volume    = {74}
}

@article{mendez2023directed,
	author = {M{\'e}ndez, David and S{\'a}nchez-Garc{\'{\i}}a, Rub{\'e}n J.},
	title = {A directed persistent homology theory for dissimilarity functions},
	fjournal = {Journal of Applied and Computational Topology},
	journal = {J. Appl. Comput. Topol.},
	issn = {2367-1726},
	volume = {7},
	number = {4},
	pages = {771--813},
	year = {2023},
	language = {English},
	doi = {10.1007/s41468-023-00124-x},
	zbMATH = {7772364},
	Zbl = {1528.55004}
}

@article{nelson2017quantum,
	author = {Nelson, Sam and Orrison, Michael E. and Rivera, Veronica},
	title = {Quantum enhancements and biquandle brackets},
	fjournal = {Journal of Knot Theory and its Ramifications},
	journal = {J. Knot Theory Ramifications},
	issn = {0218-2165},
	volume = {26},
	number = {5},
	pages = {24},
	note = {Id/No 1750034},
	year = {2017},
	language = {English},
	doi = {10.1142/S0218216517500341},
	zbMATH = {6731377},
	Zbl = {1396.57023}
}

@article{nelson2019biquandle,
	author = {Nelson, Sam and Oshiro, Kanako and Shimizu, Ayaka and Yaguchi, Yoshiro},
	title = {Biquandle virtual brackets},
	fjournal = {Journal of Knot Theory and its Ramifications},
	journal = {J. Knot Theory Ramifications},
	issn = {0218-2165},
	volume = {28},
	number = {11},
	pages = {22},
	note = {Id/No 1940003},
	year = {2019},
	language = {English},
	doi = {10.1142/S0218216519400030},
	zbMATH = {7139440},
	Zbl = {1436.57014}
}

@misc{nelson2024quandle,
	author = {Sam Nelson},
	title = {Quandle {Cohomology} {Quiver} {Representations}},
	year = {2024},
	howpublished = {Preprint, {arXiv}:2411.02153 [math.{GT}] (2024)},
	url = {https://arxiv.org/abs/2411.02153},
	arXiv = {arXiv:2411.02153}
}

@ARTICLE{reimann2017cliques,
	
	AUTHOR={Reimann, Michael W.  and Nolte, Max  and Scolamiero, Martina  and Turner, Katharine  and Perin, Rodrigo  and Chindemi, Giuseppe  and Dłotko, Paweł  and Levi, Ran  and Hess, Kathryn  and Markram, Henry },
	
	TITLE={Cliques of Neurons Bound into Cavities Provide a Missing Link between Structure and Function},
	
	JOURNAL={Frontiers in Computational Neuroscience},
	
	VOLUME={Volume 11 - 2017},
	
	YEAR={2017},
	
	URL={https://www.frontiersin.org/journals/computational-neuroscience/articles/10.3389/fncom.2017.00048},
	
	DOI={10.3389/fncom.2017.00048},
	
	ISSN={1662-5188}}

@article{hepworth2025reach,
	author = {Hepworth, Richard and Roff, Emily},
	title = {The reachability homology of a directed graph},
	fjournal = {IMRN. International Mathematics Research Notices},
	journal = {Int. Math. Res. Not.},
	issn = {1073-7928},
	volume = {2025},
	number = {3},
	pages = {18},
	note = {Id/No rnae280},
	year = {2025},
	language = {English},
	doi = {10.1093/imrn/rnae280},
	zbMATH = {8005333},
	Zbl = {1560.55012}
}

@article{turaev2012knotoids,
	author = {Turaev, Vladimir},
	title = {Knotoids},
	fjournal = {Osaka Journal of Mathematics},
	journal = {Osaka J. Math.},
	issn = {0030-6126},
	volume = {49},
	number = {1},
	pages = {195--223},
	year = {2012},
	language = {English},
	zbMATH = {6029033},
	Zbl = {1271.57030}
}

@article{zomorodian2005computing,
	author = {Zomorodian, Afra and Carlsson, Gunnar},
	title = {Computing persistent homology},
	fjournal = {Discrete \& Computational Geometry},
	journal = {Discrete Comput. Geom.},
	issn = {0179-5376},
	volume = {33},
	number = {2},
	pages = {249--274},
	year = {2005},
	language = {English},
	doi = {10.1007/s00454-004-1146-y},
	zbMATH = {2156870},
	Zbl = {1069.55003}
}

\end{document}